\documentclass[10pt]{amsart}
\usepackage[latin1]{inputenc}
\usepackage{amsmath}
\usepackage{amsfonts}
\usepackage{amssymb}
\usepackage{amsthm}
\usepackage{graphicx,epsfig}
\usepackage{amscd}
\usepackage[all]{xy}
\usepackage{hyperref}
\DeclareMathOperator{\supp}{supp}

\DeclareMathOperator{\dv}{div}

\DeclareMathOperator{\tr}{tr}

\DeclareMathOperator{\Area}{Area}

\DeclareMathOperator{\sgn}{sign}

\newcommand{\D}{\mathbb{D}}
\newcommand{\M}{\mathbb{M}}
\newcommand{\R}{\mathbb{R}}
\newcommand{\h}{\mathbb{H}}
\newcommand{\Z}{\mathbb{Z}}
\newcommand{\N}{\mathbb{N}}
\newcommand{\C}{\mathbb{C}}
\newcommand{\s}{\mathbb{S}}

\newcommand{\rmd}{\mathrm{d}}

\newcommand{\rmo}{\mathrm{o}}

\newcommand{\cM}{{\mathcal M}}

\newcommand{\cE}{{\mathcal E}}

\newcommand{\cH}{{\mathcal H}}

\newcommand{\cC}{{\mathcal C}}

\newcommand{\cL}{{\mathcal L}}

\newcommand{\cQ}{{\mathcal Q}}

\newcommand{\tra}{{}^\mathrm{t}}

\makeatletter
\@namedef{subjclassname@2020}{\textup{2020} Mathematics Subject Classification}
\makeatother

\numberwithin{equation}{section}

\begin{document}

\newtheorem{thm}{Theorem}[section]
\newtheorem*{thmintro}{Theorem}
\newtheorem{cor}[thm]{Corollary}
\newtheorem{prop}[thm]{Proposition}
\newtheorem{app}[thm]{Application}
\newtheorem{lemma}[thm]{Lemma}
\newtheorem{notation}[thm]{Notations}
\newtheorem{hypothesis}[thm]{Hypothesis}

\newtheorem{defin}[thm]{Definition}
\newenvironment{defn}{\begin{defin} \rm}{\end{defin}}
\newtheorem{remk}[thm]{Remark}
\newenvironment{rem}{\begin{remk} \rm}{\end{remk}}
\newtheorem{exa}[thm]{Example}
\newenvironment{ex}{\begin{exa} \rm}{\end{exa}}
\newtheorem{cla}[thm]{Claim}
\newenvironment{claim}{\begin{cla} \rm}{\end{cla}}

\title{Generalized Ricci surfaces}

\author[Beno\^it Daniel]{Beno\^it Daniel} 
\address{Universit\'e de Lorraine, CNRS, IECL, F-54000 Nancy, France}
\email{benoit.daniel@univ-lorraine.fr}

\author[Yiming Zang]{Yiming Zang}
\address{Department of Sciences, North China University of Technology, Beijing 100144, P.R. China}
\email{yiming.zang@ncut.edu.cn}

\thanks{B. D. is partially supported by the ANR project Min-Max (ANR-19-CE40-0014).}

\keywords{Ricci surface, curvature, isometric immersion, minimal surface}
\subjclass[2020]{Primary 53C25, 53C42; Secondary 30F45, 53A15}

\begin{abstract}
We consider smooth Riemannian surfaces whose curvature $K$ satisfies the relation $\Delta\log|K-c|=aK+b$ away from points where $K=c$ for some $(a,b,c)\in\R^3$, which we call generalized Ricci surfaces. We prove some isometric immersion theorems allowing points where $K=c$ using properties of log-harmonic functions. For instance, we obtain a characterization of Riemannian surfaces that locally admit minimal isometric immersions, possibly with umbilical points, into a $3$-dimensional Riemannian manifold of constant sectional curvature. We also give an application to convex affine spheres. Finally, we study compact generalized Ricci surfaces: we obtain topological obstructions and construct examples.
\end{abstract}

\maketitle

\tableofcontents

\section{Introduction}

In this paper we will consider smooth (i.e., of class $\cC^\infty$) Riemannian surfaces whose curvature $K$ satisfies the relation $\Delta\log|K-c|=aK+b$ away from points where $K=c$ for some $(a,b,c)\in\R^3$. We will call them \emph{generalized Ricci surfaces}.

A first motivation comes from minimal isometric immersions. For $c\in\R$, let $\M^3(c)$ be the $3$-dimensional Riemannian space form of sectional curvature $c$. By the Gauss equation, the curvature $K$ of a minimal surface in $\M^3(c)$ satisfies $K\leqslant c$. Points where $K=c$ are umbilical points. Also, it is well known \cite{ricci,lawson,docarmo-dajczer} that if $(\Sigma,\rmd s^2)$ is a Riemannian surface whose curvature $K$ satisfies $K<c$, then $(\Sigma,\rmd s^2)$ locally admits minimal isometric immersions into $\M^3(c)$ if and only if $\Delta\log|K-c|=4K$, or, equivalently, if and only if the metric $\sqrt{|K-c|}\rmd s^2$ is flat. 

The hypothesis $K<c$ can be weakened; instead, it can be assumed that $K\leqslant c$ and the function $\sqrt{|K-c|}$ is of \emph{absolute value type} \cite{egt,et-rmsup} (see Definition \ref{avt}).

When $c=0$, A. Moroianu and S. Moroianu \cite{moroianu} extend this result replacing the hypothesis $K<0$ by the hypothesis that $K\leqslant0$ and the metric is smooth. To do this, they prove that a smooth \emph{log-harmonic} function (see Definition \ref{logharmonic}) has isolated zeroes or vanishes identically. They deduce from this that if a smooth Riemannian surface satisfies $\Delta\log|K|=4K$ away from points where $K=0$ (which they call a \emph{Ricci surface}), then either $K\equiv0$ or the zeroes of $K$ are isolated. If moreover $K\leqslant0$, then they recover a local minimal isometric immersion into $\R^3$ via the spinorial Weierstrass representation. A. Moroianu and S. Moroianu also construct compact Ricci surfaces. Then, the second author of this paper \cite{zang} constructs non-compact Ricci surfaces with ``catenoidal'' ends.

In this paper we will extend this isometric immersion result for any $c\in\R$ (item \eqref{ii-min3} of Theorem \ref{immersion}).

More generally, different classes of surfaces satisfy an equation of the form $\Delta\log|K-c|=aK+b$, possibly assuming $K\neq c$: certain minimal surfaces (e.g., complex curves in complex $2$-dimensional K\"ahler space forms), gradient Ricci solitons, certain biconservative surfaces, convex affine spheres endowed with the Blaschke metric, etc.

In Section \ref{sec:generalities} we define generalized Ricci surfaces and we derive some first general properties. In particular (Theorem \ref{main}), we prove that function $\sqrt{|K-c|}$ is of absolute value type; hence, either $K\equiv c$ or the zeroes of $K-c$ are isolated. To do this, we use the properties of smooth log-harmonic maps proved by A. Moroianu and S. Moroianu. Also, we prove that multiplying the metric by a suitable power of $|K-c|$ yields, away from points where $K=c$, another generalized Ricci metric or a constant curvature metric in some cases.

In Section \ref{sec:param}, we establish useful equations in terms of a local conformal coordinate. For some particular values of $(a,b,c)$ we make a connection with Toda systems.

Section \ref{sec:geometric} is devoted to geometric properties of generalized Ricci surfaces for some particular values of $(a,b,c)$. We extend (Theorem \ref{immersion}) isometric immersions theorems that exist in the literature, considering the points where $K=c$ thanks to Theorem \ref{main}. In the same spirit, we extend a characterization of the Blaschke metric of convex affine spheres (Theorem \ref{thmaffine}). We also consider some variational properties.

Finally, we will study in Section \ref{sec:compact} compact orientable generalized Ricci surfaces. We will be mainly interested in the following question: for which values of $(a,b,c)\in\R^3$ and $g\in\N$ does their exist non constant curvature compact orientable generalized Ricci surfaces of type $(a,b,c)$ and genus $g$?

We establish necessary conditions using an integral formula for absolute value type functions by Eschenburg, Guadalupe and Tribuzy \cite{egt} and elementary considerations. In the very particular case of spheres of type $(a,0,0)$, we have an extra condition coming from the fact that these surfaces admit a rational function $\overline\C\to\overline\C$ as ``developing map''.

We also construct examples of a given genus $g$. We restrict ourselves to the case where $b=0$: then, a flat metric with point singularities or ends is naturally associated to a generalized Ricci metric. 
In many cases we are also able to prescribe some properties of the zeroes of the function $K-c$ and the conformal type of the surface.
To do that we use different techniques.
\begin{itemize}
 \item Some rotational or translational examples are constructed from solutions of an ordinary differential equation (Proposition \ref{spherea0c} and Example \ref{delaunay}).
 \item In Theorem \ref{a0chighgenus}, we look for metric in the conformal class of a hyperbolic metric on a surface of genus $g\geqslant2$. We find a conformal factor satisfying an order $2$ partial differential equation.
 \item When $c=0$ (Theorem \ref{sphere0}, Proposition \ref{sphere2} and Theorem \ref{a00highgenus}), we start with a constant curvature metric with prescribed conical singularities and we construct a generalized Ricci metric using a ``reciprocal Gauss-Bonnet theorem'' by Wallach and Warner \cite{ww} or, when $a\neq2$, adapting the method of A. Moroianu and S. Moroianu.
 \end{itemize}
 
We can summarize necessary and/or sufficient conditions that we establish in this paper for the existence of a non constant curvature compact orientable generalized Ricci surface of type $(a,0,c)$ of a given genus.
\begin{itemize}
 \item For spheres:
 \begin{itemize}
  \item if $c=0$, then a necessary and sufficient condition is $a\in-2\N^*$ (Proposition \ref{sphere} and Theorem \ref{sphere0});
  \item if $c\neq0$, then $a\in-\N^*$ is a necessary condition (Proposition \ref{sphere}) and $a\in-2\N^*$ is a sufficient condition (Proposition \ref{spherea0c}).
 \end{itemize}
 \item For tori: a necessary and sufficient condition is $a\cdot c>0$ (Proposition \ref{torus} and Example \ref{delaunay}).
 \item For surfaces of genus $g\geqslant2$: 
 \begin{itemize}
  \item if $c\leqslant0$, then a necessary and sufficient condition is $(g-1)a\in\N^*$ (Propositions \ref{highgenus}, \ref{a00highgenus} and Theorem \ref{a0chighgenus});
  \item if $c>0$, then $(g-1)a\in\N^*$ is a necessary condition (Proposition \ref{highgenus}); existence is already known for certain values of $a$ ($a=4$ for any $g$, $a=6$ for any odd $g$ and for $g=4$; see Example \ref{ex:lawson}).
 \end{itemize}
\end{itemize}

The question for $b\neq0$ remains largely open. We prove some necessary conditions but we only have examples already known in the literature or that can be deduced from them (Examples \ref{ex:veronese}, \ref{ex:torus}, \ref{ex:superminimal} and \ref{ex:superminimal2}).

We emphasize the importance of the smoothness hypothesis in this work. Non smooth surfaces satisfying $\Delta\log|K-c|=aK+b$ away from points where $K=c$ exist: see \cite[Remark 12.1]{lawson} and Remarks \ref{rknotsmooth} and \ref{spherenotsmooth}. However one cannot expect a characterization in terms of isometric immersions as in Theorem \ref{immersion} for them. Also, our results use the properties of smooth log-harmonic functions. For compact surfaces one cannot expect necessary conditions as the previous ones either. 

\section{Generalities} \label{sec:generalities}

\subsection{Terminology and notations}

\begin{itemize}
 \item In this paper, ``smooth'' means ``of class $\cC^\infty$''.
 \item All surfaces are assumed connected and without boundary.
 \item A metric on a Riemann surface is \emph{conformal} if it is compatible with the complex structure.
 \item If $(\Sigma,\rmd s^2)$ is a Riemannian surface, when there is no ambiguity we will denote by $||\,||$ its associated norm, $\mu$ its area form, $\nabla$ its gradient operator, $\Delta=\dv\nabla$ its Laplace-Beltrami operator (note that our sign convention differs from that of some references, for instance \cite{moroianu,cmop,fno,zang}), and $K$ its curvature.

We also recall that if two metrics $\rmd s_1^2$ and $\rmd s_2^2$ on a differentiable surface are related by $\rmd s_2^2=e^{-2f}\rmd s_1^2$ where $f:\Sigma\to\R$ is a smooth function, then their respective Laplace-Beltrami operators $\Delta_1$ and $\Delta_2$ and their respective curvatures $K_1$ and $K_2$ are related by
$$\Delta_2=e^{2f}\Delta_1,\quad\quad K_2=e^{2f}(K_1+\Delta_1f).$$
\item
If $n\geqslant3$ is an integer and $c\in\R$, 
\begin{itemize}
 \item $\M^n(c)$ will be the $n$-dimensional Riemannian space form of constant sectional curvature $c$ (i.e., the sphere $\s^n(c)$ if $c>0$, Euclidean space $\R^n$ if $c=0$, and hyperbolic space $\h^n(c)$ if $c<0$),
 \item $\M^n_1(c)$ will be the $n$-dimensional Lorentzian (i.e., with a metric of signature $(n-1,1)$) space form of constant sectional curvature $c$ (i.e., de Sitter space $\s^n_1(c)$ if $c>0$, Lorentz space $\R^n_1$ if $c=0$, and the universal cover of anti-de Sitter space $\h^n_1(c)$, 
 \item $\C\M^n(c)$ will be the (complex) $n$-dimensional K\"ahler space form of constant holomorphic sectional curvature $c$ (i.e., complex projective space $\C\mathbb{P}^n(c)$ if $c>0$, complex Euclidean space $\C^n$ if $c=0$, and complex hyperbolic space $\C\h^n(c)$ if $c<0$).
\end{itemize}
\item
We also let $\overline\C=\C\cup\{\infty\}$ denote the Riemann sphere.
\item
If $F$ is a real-valued function, the notation $\sgn F=1$ (respectively, $\sgn F=-1$) will mean that $F\geqslant0$ and $F\not\equiv0$ (respectively, $F\leqslant0$ and $F\not\equiv0$). 
\item A \emph{partition} of a positive integer $N$ is a tuple $(m_1,\dots,m_n)$ such that $n\in\N^*$, $m_j\in\N^*$ for all $j\in\{1,n\dots,n\}$ and $m_1+\dots+m_n=N$.
\end{itemize}

\subsection{Metrics with conical singularities}

We need to recall a few facts about metrics with conical singularities. We refer to \cite{mcowen,troyanov,troyanov1991prescribing}. A conformal metric $\rmd s^2$ on a Riemann surface $\Sigma$ is said to have a \emph{conical singularity} of order $\beta>-1$ (or of angle $2\pi(\beta+1)>0$) at a point $p\in\Sigma$ if in some neighborhood $U$ of $p$, we have $\rmd s^2=e^{2u}|\rmd z|^2$
where $z$ is a conformal coordinate in the neighborhood such that $z=0$ at $p$ and $u$ is a smooth function on $U\setminus\{0\}$ such that
$u_1:z\mapsto u(z)-\beta\log|z|$  is continuous at $0$. Then $\rmd s^2=e^{2u_1}|z|^{2\beta}|\rmd z|^2$ and the function $u_1$ satisfies $$(u_1)_{z\bar z}=-\frac K4|z|^{2\beta}e^{2u_1}$$ on $U\setminus\{0\}$ where $K$ is the curvature function. We now assume that $K$ is a constant $\kappa$. If $\beta>0$, then $u_1$ is of class $\cC^2$ at $0$. If $\beta\in\N^*$, then by induction we get $u_1\in\cC^{2k+1,\alpha}(U,\R)$ for $\alpha\in(0,1)$ and every $k\in\N$, and so $u_1\in\cC^\infty(U,\R)$. If $\kappa=0$, then $u_1\in\cC^\infty(U,\R)$ and it is harmonic. For instance, we may chose the conformal coordinate $z$ so that
$$\rmd s^2=\frac{4(\beta+1)^2|z|^{2\beta}}{(1+\kappa|z|^{2\beta+2})^2}|\rmd z|^2.$$

 In this paper we will use constructions of constant curvature metrics with conical singularities on compact surfaces from \cite{mcowen,troyanov-flat,troyanov,troyanov1991prescribing,bdmm,eremenko-gt,mondello2019spherical}.

\subsection{Definitions and first properties}

Motivated by many situations in differential geometry of surfaces, we introduce the following generalization of A. Moroianu and S. Moroianu's notion of Ricci surfaces \cite{moroianu}, which will be the object of this paper.

\begin{defin}
 Let $(a,b,c)\in\R^3$. We say that a smooth Riemannian surface $(\Sigma,\rmd s^2)$ is a \emph{generalized Ricci surface} of type $(a,b,c)$ if its curvature satisfies \begin{equation} \label{ricci}
  (c-K)\Delta K+||\nabla K||^2+(aK+b)(K-c)^2=0.
 \end{equation}
 The metric $\rmd s^2$ is called a \emph{generalized Ricci metric} of type $(a,b,c)$.
\end{defin}

In particular, generalized Ricci surfaces of type $(4,0,0)$ are Ricci surfaces in the sense of \cite{moroianu}.

\begin{lemma} \label{deltalog}
Let $(a,b,c)\in\R^3$. Let $(\Sigma,\rmd s^2)$ be a smooth Riemannian surface. Then $(\Sigma,\rmd s^2)$ is a generalized Ricci surface of type $(a,b,c)$ if and only if its curvature satisfies 
$$\Delta\log|K-c|=aK+b$$ on any open set where $K-c$ does not vanish.
\end{lemma}

\begin{proof}
 This follows from the fact that, on an open set where $K-c$ does not vanish, $\Delta\log|K-c|=\frac{\Delta K}{K-c}-\frac{||\nabla K||^2}{(K-c)^2}$, and from the fact that \eqref{ricci} clearly holds an open set where $K\equiv c$.
\end{proof}

\begin{rem} \label{kappa}
 Obviously, a surface of constant curvature $\kappa$ is a generalized Ricci surface of type $(a,b,c)$ if and only if $\kappa=c$ or $a\kappa+b=0$.
\end{rem}

\begin{rem} \label{homothety}
 If $(\Sigma,\rmd s^2)$ is a Ricci surface of type $(a,b,c)\in\R^3$ and if $r>0$, then $(\Sigma,r^2\rmd s^2)$ is a Ricci surface of type $(a,b/r^2,c/r^2)$.
\end{rem}

We recall the notions of log-harmonic function and absolute value type function, which will be crucial in this work.

\begin{defin}[see \cite{moroianu}] \label{logharmonic}
  Let $\Sigma$ be a Riemann surface and $U$ be an open subset of $\Sigma$. A smooth function $F:U\to[0,+\infty)$ is \emph{log-harmonic} if $\log|F|$ is harmonic on any open set where $F$ does not vanish.
\end{defin}

\begin{defin}[Eschenburg, Guadalupe and Tribuzy, \cite{egt}] \label{avt}
 Let $\Sigma$ be a Riemann surface. A function $F:\Sigma\to[0,+\infty)$ is of \emph{absolute value type} if any point $p\in\Sigma$ admits a neighborhood $U$ on which there exist a holomorphic funtion $h:U\to\C$ and a smooth function $r:U\to\C^*$ such that $F=|rh|$ on $U$. Then, if $F$ has an isolated zero at $p$, then the \emph{order} of $F$ at $p$ is the order of $h$ at $p$.
\end{defin}

The main general result in this paper relies on properties of log-harmonic functions proved by A. Moroianu and S. Moroianu \cite{moroianu}; it extends one of their main results about Ricci surfaces.

\begin{thm} \label{main}
 Let $(\Sigma,\rmd s^2)$ be a generalized Ricci surface of type $(a,b,c)\in\R^3$. Then, either $K\equiv c$ on $\Sigma$, or the zeroes of $K-c$ are isolated. In particular, $K-c$ does not change sign. Moreover, the function $\sqrt{|K-c|}$ is of absolute value type.
\end{thm}

\begin{proof}
 Let $U$ be a neighborhood of a given point on which there exists a function $\varphi:U\to\R$ such that $\Delta\varphi=1$. On $U$ we write the metric as $\rmd s^2=e^{-2f}|\rmd z|^2$ where $z$ is a conformal coordinate and $f:U\to\R$ is smooth. Then the function $e^{-af}e^{-b\varphi}(K-c)$ is log-harmonic on $U$; indeed, it is smooth and, on an open subset where it does not vanish, we have
 $$\Delta\log|e^{-af}e^{-b\varphi}(K-c)|=-a\Delta f-b\Delta\varphi+\Delta\log|K-c|=0$$ by Lemma \ref{deltalog} and the fact that $\Delta f=K$. Then, by \cite[Theorem 4.6]{moroianu}, either $e^{-af}e^{-b\varphi}(K-c)\equiv0$ on $U$ or the set of zeroes of $e^{-af}e^{-b\varphi}(K-c)$ is a discrete subset of $U$. 
 
 From this we deduce that the set of non-isolated zeroes of $K-c$ is open and closed in $\Sigma$. Hence, either $K\equiv c$ on $\Sigma$, or the zeroes of $K-c$ are isolated. In particular, $K-c$ does not change sign.  
 
 If $K\equiv c$, then obviously $\sqrt{|K-c|}$ is of absolute value type. We now assume that $K-c$ has isolated zeroes. Let $p\in\Sigma$ be a zero of $K-c$. Let $U$ around $p$, $f$ and $\varphi$ be as above; we may moreover assume that $K-c$ does not vanish on $U\setminus\{p\}$ and that $U$ is conformally equivalent to $\D$. Then, by \cite[Lemma 4.5]{moroianu}, there exists a holomorphic function $h$ on $U$ such that 
 $$e^{-af}e^{-b\varphi}(K-c)=\varepsilon|h|^2$$
 on $U$ with $\varepsilon=\pm1$ depending on the sign of $K-c$. Then, $\sqrt{|K-c|}=|e^{af/2}e^{b\varphi/2}h|$ on $U$. This proves that $\sqrt{|K-c|}$ is of absolute value type.
\end{proof}

\subsection{Multiplying the metric by a power of $|K-c|$}

We now show how to obtain from a generalized Ricci metric, away from points where $K=c$, constant curvature metrics or other generalized Ricci metrics by making a suitable conformal change of metric.

\begin{lemma} \label{tildek}
Let $(a,b,c)\in\R^3$ and let $(\Sigma,\rmd s^2)$ be a smooth Riemannian surface such that $K-c$ does not vanish. Let $\gamma\in\R^*$. Then $(\Sigma,\rmd s^2)$ is a generalized Ricci surface of type $(a,b,c)$ if and only if  the curvature of the metric $|K-c|^\gamma\rmd s^2$ is $$\tilde K=|K-c|^{-\gamma}\left(\left(1-\frac{\gamma a}2\right)K-\frac{\gamma b}2\right).$$ 
\end{lemma}

\begin{proof}
 We have $\tilde K=|K-c|^{-\gamma}\left(K-\frac\gamma2\Delta\log|K-c|\right)$ and we conclude by Lemma \ref{deltalog}.
\end{proof}

We now extend a result of Fetcu, Nistor and Oniciuc \cite[Proposition 3.3]{fno} concerning generalized Ricci surfaces of type $(8/3,0,c)$.

\begin{cor} \label{constantcurvature}
Let $(a,c)\in\R^2$ and let $(\Sigma,\rmd s^2)$ be a smooth Riemannian surface such that $K-c$ does not vanish.
 \begin{enumerate}
  \item If $a\neq0$, then $(\Sigma,\rmd s^2)$ is a generalized Ricci surface of type $(a,0,c)$ if and only if $|K-c|^{2/a}\rmd s^2$ is a flat metric on $\Sigma$.
  \item Let $b=(2-a)c$. Then $(\Sigma,\rmd s^2)$ is a generalized Ricci surface of type $(a,b,c)$ if and only if $|K-c|\rmd s^2$ is a metric with constant curvature $\left(1-\frac a2\right)\sgn(K-c)$ on $\Sigma$.
 \end{enumerate}
 \end{cor}

 \begin{proof}
  \begin{enumerate}
   \item Apply Lemma \ref{tildek} with $\gamma=2/a$.
   \item Apply Lemma \ref{tildek} with $\gamma=1$.
  \end{enumerate}

 \end{proof}

The following result for surfaces of type $(a,0,0)$ was already essentially observed by Bernstein and Mettler \cite{bm-jga}. 

\begin{cor} \label{conformalricci}
Let $a\in\R$. Let $(\Sigma,\rmd s^2)$ be a generalized Ricci surface of type $(a,0,0)$ such that $K$ does not vanish identically. Let $\gamma\in\R\setminus\{1\}$ such that $\gamma a\neq2$. Let $\Sigma_*=\{x\in\Sigma\mid K(x)\neq 0\}$. Then $(\Sigma_*,|K|^\gamma\rmd s^2)$ is generalized Ricci surface of type $\left(\frac{2a(1-\gamma)}{2-\gamma a},0,0\right)$ of curvature $\tilde K=\left(1-\frac{\gamma a}2\right)|K|^{-\gamma}K$.
\end{cor}

\begin{proof}
Let $\rmd\tilde s^2=|K|^\gamma\rmd s^2$ on $\Sigma_*$. Denote by $\tilde\Delta$ its Laplace-Beltrami operator and by $\tilde K$ its curvature. By Lemma \ref{tildek} with $b=c=0$ we have $\tilde K=\left(1-\frac{\gamma a}2\right)|K|^{-\gamma}K$. In particular $\tilde K$ does not vanish on $\Sigma_*$. So, using Lemma \ref{deltalog}, we get
$$\tilde\Delta\log|\tilde K|=|K|^{-\gamma}\Delta\log|\tilde K|=|K|^{-\gamma}(1-\gamma)aK=\frac{(1-\gamma)a}{1-\frac{\gamma a}2}\tilde K,$$ which proves the result.
\end{proof}

\begin{rem}
 \begin{enumerate}
  \item If $a\notin\{0,2\}$ then the range of $\frac{2a(1-\gamma)}{2-\gamma a}$ as $\gamma$ varies in $\R\setminus\left\{1,\frac2a\right\}$ is $\R\setminus\{0,2\}$. On the other hand, if $a=0$ (respectively, $a=2$), then $\frac{2a(1-\gamma)}{2-\gamma a}=0$ (respectively, $\frac{2a(1-\gamma)}{2-\gamma a}=2$) for any $\gamma\in\R\setminus\{1\}$.
  \item For $\gamma\neq1$ such that $\gamma a\neq2$, we have $|\tilde K|^{\frac\gamma{\gamma-1}}|K|^\gamma=\left|1-\frac{\gamma a}2\right|^{\frac\gamma{\gamma-1}}$, hence the metric $|\tilde K|^{\frac\gamma{\gamma-1}}|K|^\gamma\rmd s^2$ on $\Sigma_*$ is homothetic to $\rmd s^2$. This provides an involutive correspondence between generalized Ricci surfaces of type $(a,0,0)$ with non vanishing curvature and  generalized Ricci surfaces of type $\left(\frac{2a(1-\gamma)}{2-\gamma a},0,0\right)$ with non vanishing curvature, up to homotheties.
  \item If $\gamma=1$, then the conclusion of Corollary \ref{conformalricci} still holds, but then the metric $|K|\rmd s^2$ has constant curvature by Corollary \ref{constantcurvature}.
 \end{enumerate}

\end{rem}

\begin{cor} \label{conformalricci2bis}
 Let $(a,b,c)\in\R^3$ such that $b\neq(2-a)c$. Let $(\Sigma,\rmd s^2)$ be a generalized Ricci surface of type $(a,b,c)$ such that $K$ is not identically $c$.  Let $\Sigma_*=\{x\in\Sigma\mid K(x)\neq c\}$. Let $\varepsilon=\sgn(K-c)$ (which is constant by Theorem \ref{main}). Then $(\Sigma_*,|K-c|\rmd s^2)$ is generalized Ricci surface of type $$\left(\frac{2(ac+b)}{b+(a-2)c},-\frac{2\varepsilon b}{b+(a-2)c},\varepsilon\left(1-\frac a2\right)\right),$$ and its curvature $\tilde K$ satisfies 
  \begin{equation} \label{dualk}
  \tilde K-\varepsilon\left(1-\frac a2\right)=-\frac{b+(a-2)c}2\varepsilon(K-c)^{-1}.
 \end{equation}
\end{cor}

\begin{proof}
 Let $\rmd\tilde s^2=|K-c|\rmd s^2$ on $\Sigma_*$. Denote by $\tilde\Delta$ its Laplace-Beltrami operator and by $\tilde K$ its curvature.  By Lemma \ref{tildek} with  $\gamma=1$ we have $$\tilde K=|K-c|^{-1}\left(\left(1-\frac a2\right)K-\frac b2\right)=\varepsilon(K-c)^{-1}\left(\left(1-\frac a2\right)K-\frac b2\right),$$ which gives \eqref{dualk}.
 In particular $\tilde K-\varepsilon\left(1-\frac a2\right)$ does not vanish on $\Sigma_*$. So, using Lemma \ref{deltalog}, we get
 \begin{eqnarray*}
  \tilde\Delta\log\left|\tilde K-\varepsilon\left(1-\frac a2\right)\right| & = &
|K-c|^{-1}\Delta\log\left|\tilde K-\varepsilon\left(1-\frac a2\right)\right| \\
& = & -|K-c|^{-1}\Delta\log|K-c| \\
& = & -\varepsilon(K-c)^{-1}(aK+b) \\
& = & \frac{2(ac+b)}{b+(a-2)c}\tilde K-\frac{2\varepsilon b}{b+(a-2)c},
 \end{eqnarray*}
 which proves the result.
\end{proof}

\begin{rem} \label{rkdual}
 If we apply Corollary \ref{conformalricci2bis} twice, then we get the intial metric up to a homothety; indeed by \eqref{dualk} we have
 $$\left|\tilde K-\varepsilon\left(1-\frac a2\right)\right||K-c|\rmd s^2=\left|\frac{b+(a-2)c}2\right|\rmd s^2.$$
\end{rem}

We end this section by generalizations of results by A. Moroianu and S. Moroianu \cite{moroianu}, which will be very useful to construct compact generalized Ricci surfaces of type $(a,0,0)$.
 
\begin{prop} \label{propV} 
    Let $\rmd\sigma_0^2$ be a flat metric on a Riemannian surface $\Sigma$ and $V:\Sigma\to\R_+^*$ be a smooth function such that $V\rmd\sigma_0^2$ is a metric of constant curvature $\kappa\neq 0$. Then $\rmd s^2_+=V^{-\frac1{|\kappa|}}\rmd\sigma_0^2$ is a generalized Ricci metric of type $(2+2|\kappa|,0,0)$ with curvature $-\sgn(\kappa)V^{\frac{|\kappa|+1}{|\kappa|}}$, and $\rmd s^2_-=V^{\frac1{|\kappa|}}\rmd\sigma_0^2$ is a generalized Ricci metric of type $(2-2|\kappa|, 0, 0)$ with  curvature $\sgn(\kappa)V^{\frac{|\kappa|-1}{|\kappa|}}$.
\end{prop}

\begin{proof}
     We let $\Delta_0$ and $\Delta_+$ denote the Laplacian-Beltrami operators of $\rmd\sigma_0^2$ and $\rmd s^2_+$ respectively, and $K_+$ the curvature of $\rmd s^2_+$. Since
    $$\kappa = -V^{-1}\Delta_0\left(\frac12\log{V}\right)$$
    and 
    $$K_+ = V^{\frac{1}{|\kappa|}}\Delta_0\left(\frac{1}{2|\kappa|}\log{V}\right)$$
     we get
    $$K_+=-\sgn(\kappa)V^{\frac{|\kappa|+1}{|\kappa|}}.$$

  Newt, we compute that
  \begin{eqnarray*}
   \Delta_+\log |K_+| & = & V^{\frac{1}{|\kappa|}}\Delta_0\log |K_+|
       = \frac{|\kappa|+1}{|\kappa|}V^{\frac{1}{|\kappa|}}\Delta_0\left(\log{V}\right) \\
    &   =  &\frac{|\kappa|+1}{|\kappa|}V^{\frac{1}{|\kappa|}}\cdot (-2\kappa V)
       =2(|\kappa|+1)K_+.
  \end{eqnarray*}
    This completes the proof for the metric $\rmd s^2_+$. The properties of $\rmd s^2_-$ can be proved in the same way.
\end{proof}

In order to get generalized Ricci metrics of type $(a,0,0)$ whose curvature may vanish, we need to consider constant curvature metrics with conical singularities.

\begin{prop} \label{propVconical}
Let $a\in(0,2)\cup(2,+\infty)$ and $m\in\N^*$. Let $U\subset\C$ be an open set containing $0$. 

Let $v\in\cC^\infty(U,\R)$ be a function such that $\rmd\sigma^2=e^{2v}|z|^{2m}|\rmd z|^2$ is a metric of constant curvature $\frac a2-1$ on $U\setminus\{0\}$ (hence with a conical singularity of order $m$ at $0$). 

Let $u\in\cC^\infty(U,\R)$ be a function such that $\rmd\sigma_0^2=e^{2u}|z|^{\frac{4m}a}|\rmd z|^2$ is a flat metric on $U\setminus\{0\}$ (hence with a conical singularity of order $\frac{2m}a$ at $0$). 

Let $V\in\cC^\infty(U\setminus\{0\},\R)$ be the function such that $\rmd\sigma^2=V\rmd\sigma^2_0$. Then $\rmd s^2=V^{\frac{2}{2-a}}\rmd\sigma^2_0$ extends to a generalized Ricci metric on $U$ of type $(a,0,0)$ whose curvature $K$ is such that $K<0$ on $U\setminus\{0\}$ and $\sqrt{|K|}$ has a zero of order $m$ at $0$.
\end{prop}

\begin{proof}
Let $\kappa=\frac a2-1$. By Proposition \ref{propV} (consider $\rmd s^2_+$ if $a>2$ and $\rmd s^2_-$ if $a\in(0,2)$), $\rmd s^2$ is a generalized Ricci metric on $U\setminus\{0\}$ of type $(a,0,0)$. We will now check that $\rmd s^2$ extends smoothly at $0$.

On $U\setminus\{0\}$, the function $V$ can be determined as $$V=e^{2v-2u}|z|^{\left(2-\frac4a\right)m}.$$
   So, a direct computation shows that $$\rmd s^2 = e^{\frac4{2-a}(v-au)}|\rmd z|^2$$ on $U\setminus\{0\}$. 
   
    Therefore, since $u$ and $v$ are smooth at the origin, $\rmd s^2$ extends smoothly on $U$. 
    
    Finally, by Proposition \ref{propV}, the curvature of $\rmd s^2$ on $U\setminus\{0\}$ is $$K=-V^{\frac a{a-2}}<0.$$ Considering the above expression of $V$ on $U\setminus\{0\}$, we obtain that $K$ has a zero of order $2m$ at $0$. This concludes the proof.
\end{proof}

\begin{prop} \label{propVconical2}
Let $a<0$ and $m\in\N^*$. Let $U\subset\C$ be an open set containing $0$. 

Let $v\in\cC^\infty(U,\R)$ be a function such that $\rmd\sigma^2=e^{2v}|z|^{2m}|\rmd z|^2$ is a metric of constant curvature $1-\frac a2$ on $U\setminus\{0\}$ (hence with a conical singularity of order $m$ at $0$). 

Let $u\in\cC^\infty(U,\R)$ be a function such that $\rmd\sigma_0^2=e^{2u}|z|^{\frac{4m}a}|\rmd z|^2$ is a flat metric on $U\setminus\{0\}$. 

Let $V\in\cC^\infty(U\setminus\{0\},\R)$ be the function such that $\rmd\sigma^2=V\rmd\sigma^2_0$. Then $\rmd s^2=V^{\frac{2}{2-a}}\rmd\sigma^2_0$ extends to a generalized Ricci metric on $U$ of type $(a,0,0)$ whose curvature $K$ is such that $K>0$ on $U\setminus\{0\}$ and $\sqrt K$ has a zero of order $m$ at $0$.
\end{prop}

\begin{proof}
Same as the proof of Proposition \ref{propVconical} with $\kappa=1-\frac a2$. Here, the formula for the curvature is $K=V^{\frac a{a-2}}$.  
\end{proof}

\begin{rem} \label{rknotsmooth}
 If we replace $m\in\N^*$ in Proposition \ref{propVconical} or \ref{propVconical2} by $\beta>0$, then $v$ is not of class $\cC^\infty$ in general but at least of class $\cC^2$; then $\rmd s^2$ is a generalized Ricci metric on $U\setminus\{0\}$ of type $(a,0,0)$ and it extends to a $\cC^2$ metric on $U$ but not necessarily smooth on $U$.
\end{rem}

\section{Conformal parametrizations} \label{sec:param}

In this section we establish some useful equations using a conformal coordinate. We then establish a connection with Toda systems in some particular cases.

\subsection{Generalities} \label{sec:conformal}

It is useful to state the following result, which essentially follows from the proof of Theorem \ref{main}.

\begin{lemma} \label{lemmaholo}
 Let $(a,c)\in\R^2$ and $\varepsilon\in\{-1,1\}$. 
 \begin{enumerate}
  \item Let $b\in\R$. Let $(\Sigma,\rmd s^2)$ be a smooth Riemannian surface. Then $(\Sigma,\rmd s^2)$ is a generalized Ricci surface of type $(a,b,c)$ with $K\equiv c$ or $\sgn(K-c)=\varepsilon$ if and only if each point of $\Sigma$ admits an open neighborhood $U$ biholomorphic to an open subset of $\C$ with the following property: if $z$ is a conformal coordinate on $U$ and $\rmd s^2=e^{-2f}|\rmd z|^2$ on $U$ with $f:U\to\R$ smooth, then there exist
  \begin{itemize}
   \item a holomorphic function $h:U\to\C$,
   \item and, when $b\neq0$, a smooth function $\varphi:U\to\R$
  \end{itemize}
such that the following equations hold on $U$: 
  \begin{equation} \label{holomorphic}
 e^{-af}e^{-b\varphi}(K-c)=\varepsilon|h|^2.
  \end{equation}
  and, when $b\neq0$, $$\Delta\varphi=1.$$
  \item  Let $U\subset\C$ be a simply connected open subset and let $f:U\to\R$ be a smooth function. Then $e^{-2f}|\rmd z|^2$ is a generalized Ricci metric on $U$ of type $(a,0,c)$ with $K\equiv c$ or $\sgn(K-c)=\varepsilon$ if only if there exists a holomorphic function $h:U\to\C$ 
  such 
  \begin{equation} \label{holomorphic0}
 e^{-af}(K-c)=\varepsilon|h|^2
  \end{equation}
  on $U$. If this holds, then the holomorphic function $h$ is unique up to multiplication by a complex number of modulus $1$.
 \end{enumerate}
\end{lemma}

\begin{proof}
 The direct implication in (1) follows from the proof of Theorem \ref{main} and the converse implication is a direct verification. The converse implication in (2) is a particular case of the converse implication in (1) and the last affirmation is obvious, so it remains to prove the direct implication in (2). The case where $K\equiv c$ is obvious with $h\equiv0$.
 
 We assume that $U\subset\C$ is a simply connected open subset, $f:U\to\R$ is a smooth function and $e^{-2f}|\rmd z|^2$ is a generalized Ricci metric on $U$ of type $(a,0,c)$ with $\sgn(K-c)=\varepsilon$. As seen in the proof of Theorem \ref{main}, the function $\sqrt{e^{-af}(K-c)}$ is log-harmonic and of absolute value type on $U$. Hence by \cite[Lemma 3.12]{egt} there exists a holomorphic function $h:U\to\C$ such that $\sqrt{e^{-af}(K-c)}=|h|$ on $U$, which implies \eqref{holomorphic0}.
\end{proof}

On an open set $U\subset\C$, since $K=4e^{2f}f_{z\bar z}$, the equations of Lemma \ref{lemmaholo} (1) read 
\begin{equation} \label{system}
\left\{\begin{array}{ccl} 
f_{z\bar z} & = & \displaystyle{\frac c4e^{-2f}+\frac\varepsilon4|h|^2e^{(a-2)f+b\varphi}} \\
\varphi_{z\bar z} & = & \displaystyle{\frac14e^{-2f}}
         \end{array}
\right.
\end{equation}

If we make a conformal change of variable $z=\psi(w)$ where $\psi:V\to U$ is a biholomorphism, writing $\rmd s^2=e^{-2F}|\rmd w|^2$ we have
$$e^{-f\circ\psi}|\psi'|=e^{-F},$$
$$e^{2F}F_{w\bar w}=e^{2f\circ\psi}f_{z\bar z}\circ\psi,$$
$$\left\{\begin{array}{ccl} 
F_{w\bar w} & = & \displaystyle{\frac c4e^{-2F}+\frac\varepsilon4|\psi'|^a|h\circ\psi|^2e^{(a-2)F+b\varphi\circ\psi}} \\
(\varphi\circ\psi)_{w\bar w} & = & \displaystyle{\frac14e^{-2F}}
         \end{array}
\right.$$ Hence, if $U$ is simply connected, this yields the changes 
\begin{equation} \label{changevar}
(f,\varphi,h)\longrightarrow(f\circ\psi-\log|\psi'|,\varphi\circ\psi,(\psi')^{a/2}(h\circ\psi)).
\end{equation}

If we set $u_1=-2f$ and $u_2=(a-2)f+b\varphi$, then \eqref{system} becomes
\begin{equation} \label{system2}
\left\{\begin{array}{ccl} 
(u_1)_{z\bar z} & = & \displaystyle{-\frac c2e^{u_1}-\frac\varepsilon2|h|^2e^{u_2}} \\
(u_2)_{z\bar z} & = & \displaystyle{\frac{b+(a-2)c}4e^{u_1}+\frac{\varepsilon(a-2)}4|h|^2e^{u_2}}
         \end{array}
\right.
\end{equation}

If $b\neq0$, then \eqref{system2} is equivalent to \eqref{system}.

If $b=0$, then it is useless to consider the function $\varphi$ and the second equation in \eqref{system}; on the other hand, the two lines of \eqref{system2} are proportional; we may choose $u_2=\left(1-\frac a2\right)u_1$, and then the first equation in \eqref{system2} is equivalent to the first equation in \eqref{system}.

We remark that $\rmd s^2=e^{u_1}|\rmd z|^2$ and $|K-c|\rmd s^2=|h|^2e^{u_2}|\rmd z|^2$ (this metric is considered in Corollaries \ref{constantcurvature} and \ref{conformalricci2bis}).

\begin{rem} \label{rkholodiff}
If $b=0$, $a\in2\N^*$ and $\Sigma$ is simply connected, then Lemma \ref{lemmaholo} and the formulas for a conformal change of variable show that $h\rmd z^{a/2}$ defines a holomorphic differential of weight $a/2$ (holomorphic $a/2$-differential) on $\Sigma$; it is unique up to multiplication by a complex number of modulus $1$. In particular, if $\Sigma$ is a sphere it vanishes identically and so $K\equiv c$ (see also Proposition \ref{sphere}).
\end{rem}

\begin{rem} \label{rkxi}
Assume that $U$ is simply connected, $a\neq0$ and that $h$ does not vanish on $U$. If $\xi\in\R^*$ has the same sign as $\varepsilon$, then we may choose $\psi$ such that $|(\psi')^{a/2}(h\circ\psi)|^2=\varepsilon\xi$. So, away from the zeroes of $K-c$, it is possible locally to choose a conformal coordinate so that  $|h|^2=\varepsilon\xi$.
\end{rem}

\begin{rem} \label{rktorus}
 Assume that $\rmd s^2$ is a generalized Ricci metric of type $(a,0,c)$ on a torus $\C/\Gamma$, where $\Gamma\subset\C$ is a lattice. Then we can lift the metric $\rmd s^2$ on $\C$, and then Lemma \ref{lemmaholo} gives the existence of a holomorphic function $h$ on $\C$ such that \eqref{holomorphic0} holds on $\C$. Then $|h|$ is doubly periodic; hence $h$ is constant. This constant is $0$ if and only if $c=0$.
\end{rem}

\subsection{Link with Toda systems} \label{toda}

For some particular values of $a$, $b$, $c$ and $\varepsilon$ and a suitable choice of local conformal coordinate away from zeroes of $K-c$, \eqref{system2} turns out to be a \emph{Toda system}. There is a huge literature about Toda systems, in particular from the point of view of integrable systems. Minimal surfaces, and more generally harmonic maps, are related to Toda systems in many situations; see for instance \cite{bpw}. We also refer to \cite[Section 6.2]{yang} and \cite[Chapter III, Section C]{dunne} for an introduction to Toda systems.

We use the notations of Section \ref{sec:conformal} and we assume that $a\notin\{0,2\}$, $c\neq0$ and $\varepsilon=\sgn\frac c{2-a}$. Around a point where $K-c\neq0$, as explained in Remark \ref{rkxi} we may consider a conformal coordinate on an open neighborhood $U$ so that $|h|^2=\varepsilon\xi$ with $\xi=\frac{2c}{2-a}$. Then \eqref{system2} reads
$$\left(\begin{array}{c} {(u_1)}_{z\bar z} \\ {(u_2)}_{z\bar z} \end{array}\right)=
-\frac c4\left(\begin{array}{cc} 2 & \frac4{2-a} \\ 2-a-\frac bc & 2 \end{array}\right)\left(\begin{array}{c} e^{u_1} \\ e^{u_2} \end{array}\right).$$ (Note that the matrix of this system is invertible if and only if $b\neq0$.)

It is a classical Toda system of rank $2$ when the coefficient matrix is proportional to a rank $2$ Cartan matrix (the definitions in certain references also require $c>0$), i.e., one of the following: 
\begin{itemize}
\item $\mathrm{A}_1\times\mathrm{A}_1=\left(\begin{array}{cc} 2 & 0 \\ 0 & 2 \end{array}\right)$, which is impossible here,
 \item $\mathrm{A}_2=\left(\begin{array}{cc} 2 & -1 \\ -1 & 2 \end{array}\right)$, i.e., $a=6$ and $b=-3c$,
  \item $\mathrm{B}_2=\left(\begin{array}{cc} 2 & -2 \\ -1 & 2 \end{array}\right)$, i.e., $a=4$ and $b=-c$,
   \item $\tra\mathrm{B}_2=\left(\begin{array}{cc} 2 & -1 \\ -2 & 2 \end{array}\right)$, i.e., $a=6$ and $b=-2c$,
    \item $\mathrm{G}_2=\left(\begin{array}{cc} 2 & -1 \\ -3 & 2 \end{array}\right)$, i.e., $a=6$ and $b=-c$,
     \item $\tra\mathrm{G}_2=\left(\begin{array}{cc} 2 & -3 \\ -1 & 2 \end{array}\right)$, i.e., $a=\frac{10}3$ and $b=-\frac c3$.
\end{itemize}
 We refer to \cite{leznov} and \cite[Section 6.2.2]{yang} for a method of resolution in these cases.
 
 In the same way, it is an affine Toda system of rank $1$ when the coefficient matrix is proportional to a rank $1$ affine Cartan matrix (the definitions in certain references also require $c>0$), i.e., one of the following: 
\begin{itemize}
 \item $\mathrm{A}_1^{(1)}=\left(\begin{array}{cc}
         2 & -2 \\
         -2 & 2
        \end{array}\right)$, i.e., $a=4$ and $b=0$,
  \item $\mathrm{A}_2^{(2)}=\left(\begin{array}{cc}
         2 & -1 \\
         -4 & 2
        \end{array}\right)$, i.e., $a=6$ and $b=0$,
   \item $\tra\mathrm{A}_2^{(2)}=\left(\begin{array}{cc}
         2 & -4 \\
         -1 & 2
        \end{array}\right)$, i.e., $a=3$ and $b=0$.
\end{itemize}

For $\mathrm{A}_1^{(1)}$ and $c>0$,  when $u_2=-u_1$, the system yields the elliptic sinh-Gordon equation $\omega_{z\bar z}+\frac12\sinh(2\omega)=0$ with $\omega=\frac12(u_1+\log c)$. This equation was used by Pinkall and Sterling \cite{ps} to classify constant mean curvature tori in $\R^3$.

In the same way, for $\mathrm{A}_2^{(2)}$ and $c>0$, when $u_2=-2u_1-\log(c^3/16)$, the system yields the elliptic Tzitzeica equation $\omega_{z\bar z}=e^{-2\omega}-e^\omega$ with $\omega=u_1+\log(c/2)$. This equation was used by H. Ma and Y. Ma \cite{mama} to classify totally real minimal tori in $\C\mathbb{P}^2(4c)$.

We refer to Section \ref{sec:immersions} for properties of some of these generalized Ricci surfaces related to isometric immersions.

\section{Some geometric properties} \label{sec:geometric}

In this section we give some geometric properties of generalized Ricci surfaces for particular values of $(a,b,c)$, extending some previously known results.

\subsection{Isometric immersions} \label{sec:immersions}

There are several intrinsic characterizations of the metrics of certain minimal, maximal or constant mean curvature surfaces in terms of Ricci-like conditions assuming either that the curvature $K$ does not attain some extremal value $c$ (coming from the Gauss equation) or that the function $\sqrt{|K-c|}$ is of absolute value type
\cite{ricci,lawson,docarmo-dajczer,tg-rmsup,et-rmsup,egt,cdvv,ev}. Based on these results, Theorem \ref{main} will allow us to extend them, removing this assumption on the curvature.

A \emph{pseudoholomorphic} (or \emph{almost complex}) immersion from a Riemann surface into the standard nearly K\"ahler $\s^6(c)$ is an immersion whose differential at any point is complex linear. 

A conformal minimal immersion into $\s^n(c)$ is \emph{superminimal} if all its generalized Hopf differentials vanish; we refer to \cite{ev} and references therein for precise definitions. Superminimal pseudoholomorphic immersions into the standard nearly K\"ahler $\s^6(c)$ are also said to have \emph{null-torsion} \cite{bryant-octonions}.

We now establish this isometric immersion theorem. Items \eqref{ii-min3} and \eqref{ii-max3} for $c=0$ and $H=0$ are due to A. Moroianu and S. Moroianu \cite[Theorem 1.2]{moroianu}.

\begin{thm} \label{immersion}
 Let $(\Sigma,\rmd s^2)$ be a smooth simply connected Riemannian surface. Let $c\in\R$.
 \begin{enumerate}
  \item \label{ii-min3} Let $H\in\R$. There exists a constant mean curvature $H$ isometric immersion $\Sigma\to\M^3(c)$ if and only if $(\Sigma,\rmd s^2)$ is a generalized Ricci surface of type $(4,0,c+H^2)$ with $K\leqslant c+H^2$.
   \item \label{ii-max3} Let $H\in\R$. There exists a constant mean curvature $H$ spacelike isometric immersion $\Sigma\to\M^3_1(c)$ if and only if $(\Sigma,\rmd s^2)$ is a generalized Ricci surface of type $(4,0,c-H^2)$ with $K\geqslant c-H^2$.
 \item \label{ii-supermin4} There exists a superminimal isometric immersion $\Sigma\to\M^4(c)$ if and only if $(\Sigma,\rmd s^2)$ is a generalized Ricci surface of type $(6,-2c,c)$ with $K\leqslant c$.
   \item \label{ii-complex2} There exists a complex isometric immersion $\Sigma\to\C\M^2(4c)$ if and only if $(\Sigma,\rmd s^2)$ is a generalized Ricci surface of type $(6,-12c,4c)$ with $K\leqslant4c$.
  \item \label{ii-lagrangian2} There exists a totally real (equivalently, Lagrangian) minimal isometric immersion $\Sigma\to\C\M^2(4c)$ if and only if $(\Sigma,\rmd s^2)$ is a generalized Ricci surface of type $(6,0,c)$ with $K\leqslant c$.
      \item \label{ii-pseudo6} Assume that $c>0$. There exists a superminimal pseudoholomorphic isometric immersion from $\Sigma$ into the standard nearly K\"ahler $\s^6(c)$ if and only if $(\Sigma,\rmd s^2)$ is a generalized Ricci surface of type $(6,-c,c)$ with $K\leqslant c$.
 \end{enumerate}
\end{thm}

\begin{proof}
We first notice that all surfaces and immersions considered in these statements are smooth.
 \begin{enumerate}
  \item This is a consequence of Theorem \ref{main} and \cite[Theorem 0]{et-rmsup} (totally umbilical surfaces correspond to the case where $K\equiv c$). We also provide a short proof for sake of self-containedness.

In a local conformal coordinate $z$, the Gauss-Codazzi equations for an isometric immersion $\Sigma\to\M^3(c)$ with first and second fundamental form $$\rmd s^2=e^{-2f}|\rmd z|^2,\quad\mathrm{II}=\cQ\rmd z^2+\cH e^{-2f}|\rmd z|^2+\overline\cQ\overline{\rmd z^2}$$ read
$$K=c+\cH^2-4e^{4f}|\cQ|^2,$$
$$\cQ_{\bar z}=\frac12e^{2f}\cH_z$$ where $K=4e^{2f}f_{z\bar z}$ is the curvature. Here, $\cH$ is the mean curvature function.

We may assume that $\Sigma$ is an open subset of $\overline\C$. If $\rmd s^2$ is a generalized Ricci metric of type $(4,0,c+H^2)$ with $K\leqslant c+H^2$,
then, by Lemma \ref{lemmaholo} (with $\varepsilon=-1$) and Remark \ref{rkholodiff}, there exists a holomorphic quadratic differential $h\rmd z^2$ such that $e^{-4f}(K-c-H^2)=-|h|^2$.
So, the Gauss-Codazzi equations are satisfied with $\cH\equiv H$ and $\cQ=h/2$. This proves the existence of the required isometric immersion.

The converse is proved in the same way.
  \item Same as \eqref{ii-min3} using the Gauss-Codazzi equations for a spacelike isometric immersion $\Sigma\to\M^3_1(c)$ (see \cite{palmer}):
  $$K=c-\cH^2+4e^{4f}|\cQ|^2,$$
$$\cQ_{\bar z}=\frac12e^{2f}\cH_z,$$ and Lemma \ref{lemmaholo} with $\varepsilon=1$.
  \item This is a consequence of Theorem \ref{main} and \cite[Theorem 3]{et-rmsup} (totally geodesic surfaces correspond to the case where $K\equiv c$).
  \item This is a consequence of Theorem \ref{main} and \cite[Theorem 3.6]{egt} (this theorem is stated for $c=1$ but, as explained in the introduction and the note added in proof in \cite{egt}, it can be generalized to any K\"ahler $4$-manifold with constant holomorphic sectional curvature, i.e., for any value of $c\in\R$).
  \item Same as \eqref{ii-complex2} using \cite[Theorem 3.8]{egt}. 
    \item By scaling (see Remark \ref{homothety}) we may assume that $c=1$. Then, this is a consequence of Theorem \ref{main} and \cite[Corollary 11.3]{ev}. Note that in our case the map is an immersion because we are assuming that the metric $\rmd s^2$ has no singularity.
 \end{enumerate}
\end{proof}

\begin{rem}
Generalized Ricci surfaces of type $(6,0,0)$ with $K\leqslant0$ appear in items \eqref{ii-supermin4}, \eqref{ii-complex2} and \eqref{ii-lagrangian2} for $c=0$, but then the statements are equivalent (see also the discussion following Theorem 3 in \cite{et-rmsup}). Indeed, if $\Sigma$ is a Riemannian surface and $X:\Sigma\to\R^4$ is a minimal isometric immersion, then the following properties are equivalent:
\begin{itemize}
\item $X$ is superminimal,
 \item there exists a complex structure on $\R^4$ for which $X$ is complex,
 \item there exists a complex structure on $\R^4$ for which $X$ is Lagrangian.
\end{itemize}
We refer to \cite[Section 3]{friedrich} for the equivalence between the first two properties. The equivalence between the last two properties is \cite[Theorem 5]{chenmorvan} (the two complex structures are actually orthogonal; see \cite[Section 2.6.2]{loftinmcintosh}). These properties are also equivalent to $X$ being special Lagrangian for a rotated complex structure; see also \cite[Section 2]{joyce-clay}.
\end{rem}

\begin{rem}
Let $c>0$. Smooth almost complex curves in an equator sphere $\s^5(c)$ of the standard nearly K\"ahler $\s^6(c)$ are locally isometric to totally real minimal surfaces in $\C\mathbb{P}^2(4c)$ (via horizontal lifts for the Hopf fibration $\s^5(c)\to\C\mathbb{P}^2(4c)$ \cite[Sections 6 and 7]{bvw}). Together with item \eqref{ii-lagrangian2}, this provides another characterization of generalized Ricci surfaces of type $(6,0,c)$ with $K\leqslant c$ (which can also be obtained from Theorem \ref{main} and \cite[Theorem 12.2]{ev}).
\end{rem}

 \begin{rem}
 The types of generalized Ricci surfaces that appear in Theorem \ref{immersion} when $c\neq0$ are related with Toda systems; see Section \ref{toda}.
\end{rem}

\begin{rem}
For $c\in\R$, a biconservative surface in $\M^3(c)$ for which the gradient of the mean curvature does not vanish is a generalized Ricci surface of type $(8/3,0,c)$ (see \cite[Theorem 3.1]{cmop} and \cite[Theorem 2.2]{fno}). However, not all generalized Ricci surfaces of type $(8/3,0,c)$ can be locally isometrically immersed as a biconservative surface in $\M^3(c)$ (see \cite[Theorem 4.5]{fno}).
\end{rem}

\begin{rem}
 It would be interesting to investigate whether there exist hyperbolic (for $c<0$) and semi-Riemannian (for $K\geqslant c$) analogues of some of these results.
\end{rem}

\subsection{An application to convex affine spheres}

An \emph{affine sphere} in the affine space $\R^3$ is a smooth surface whose affine normals either all meet in an point or are all parallel. We refer for instance to \cite{loftin-survey,loftinmcintosh} and references therein for surveys on this subject. Note that an affine sphere is not necessarily a topological sphere. The affine second fundamental form of a convex affine sphere is positive definite; it is called the \emph{affine metric} or the \emph{Blaschke metric}.

Opozda \cite{opozda} obtained an intrinsic characterization of the Blaschke metric of affine spheres with non vanishing Pick invariant in terms of a Ricci-like condition. In the convex case, we extend this theorem without assumption on the Pick invariant. For this purpose, we use the integrability equations of Loftin and McIntosh \cite{loftinmcintosh} in terms of a Tzitzeica-type equation, which is actually equivalent to the first equation in \eqref{system} with $a=6$, $b=0$ and $\varepsilon=1$.

\begin{thm} \label{thmaffine}
 Let $(\Sigma,\rmd s^2)$ be a smooth simply connected Riemannian surface. Then there exists an immersion from $\Sigma$ into $\R^3$ as a convex affine sphere with Blaschke metric $\rmd s^2$ if and only if there exists $c\in\R$ such that $(\Sigma,\rmd s^2)$ is a generalized Ricci surface of type $(6,0,c)$ with $K\geqslant c$. Morover this affine sphere is hyperbolic, parabolic or elliptic if $c<0$, $c=0$ or $c>0$ respectively.
\end{thm}

\begin{proof}
Let $c\in\R$ and let $(\Sigma,\rmd s^2)$ be a simply connected generalized Ricci surface of type $(6,0,c)$. On an open simply connected set $U\subset\C$ we write $\rmd s^2=e^{-2f}|\rmd z|^2$ where $f$ is a smooth function. Then by Lemma \ref{lemmaholo} we have $e^{-6f}(K-c)=|h|^2$ where $h\rmd z^3$ is a holomorphic cubic differential (see Remark \ref{rkholodiff}). Then \cite[equation (2.2)]{loftinmcintosh} is satisfied with $\psi=-f$, $Q=h/\sqrt2$ and $\lambda=c/2$. Then by \cite[Theorem 2.4]{loftinmcintosh} there exists an immersion from $U$ into $\R^3$ as a convex affine sphere such that $\rmd s^2$ is the induced Blaschke metric, and this affine sphere is hyperbolic, parabolic or elliptic if $c<0$, $c=0$ or $c>0$ respectively (note that we need not do the normalization $\lambda\in\{-1,0,1\}$ made in \cite{loftinmcintosh}). As $\Sigma$ is simply connected, this proves the existence of the required immersion.

The converse follows from the fact that \cite[equation (2.2)]{loftinmcintosh} is also necessary.
\end{proof}

\subsection{A variational property of surfaces of type $(-2,b,0)$}

Let $\Sigma$ be a smooth differentiable surface. On the space $\cM_*(\Sigma)$ of smooth Riemannian metrics on $\Sigma$ with non-vanishing curvature, we consider the functional $$\cE(\rmd s^2)=\int_\Sigma K\log|K|\,\mu$$ where $K$ is the curvature of the metric $\rmd s^2$ and $\mu$ its area form. This functional is related to Ricci flow on surfaces (see \cite{bm-jga,bm-imrn} and references therein).

Item (1) of the following result is due to Bernstein and Mettler \cite{bm-jga}.

\begin{prop} \label{stationary}
 Let $\rmd s^2\in\cM_*(\Sigma)$.
 \begin{enumerate}
  \item The metric $\rmd s^2$ is stationary for $\cE$ with respect to compactly supported conformal deformations of the metric if and only if it is a generalized Ricci metric of type $(-2,0,0)$.
  \item The metric $\rmd s^2$ is stationary for $\cE$ with respect to compactly supported conformal deformations of the metric that preserve the total area if and only if it is a generalized Ricci metric of type $(-2,b,0)$ for some $b\in\R$.
 \end{enumerate}
\end{prop}

\begin{proof} We only prove item (2).
Let $\eta$ be a smooth symmetric $2$-form on $\Sigma$ with compact support. Let $\delta>0$ and $(\rmd s^2_t)_{t\in(-\delta,\delta)}$ be a smooth family of metrics on $\Sigma$ such that
$$\rmd s^2_t=\rmd s^2+t\eta+\rmo(t)$$ as $t\to0$.

Following \cite{bm-jga,bm-imrn} we have
$$\left.\frac\rmd{\rmd t}\right|_{t=0}\cE(\rmd s^2_t)=-\frac14\int_\Sigma\left(H(\Delta\log|K|+2K)-2\langle\eta,\mathring\nabla^2\log|K|\rangle\right)\mu$$ where $\langle\,,\,\rangle$ is the natural bilinear pairing associated to $\rmd s^2$, $H=\tr\eta$ and $\mathring\nabla^2$ is the trace-free part of the Hessian with respect to $\rmd s^2$. If moreover $(\rmd s^2_t)_{t\in(-\delta,\delta)}$ is a conformal deformation, then $\langle\eta,\mathring\nabla^2\log|K|\rangle=0$ and so
$$\left.\frac\rmd{\rmd t}\right|_{t=0}\cE(\rmd s^2_t)=-\frac14\int_\Sigma H(\Delta\log|K|+2K)\mu.$$

Also, if $\mu_t$ denotes the area form for the metric $\rmd s^2_t$, then
$$\left.\frac\rmd{\rmd t}\right|_{t=0}\mathrm{Area}_{\rmd s^2_t}(\supp\eta)=\left.\frac\rmd{\rmd t}\right|_{t=0}\int_{\supp\eta}\mu_t=\frac12\int_\Sigma H\mu.$$

Hence, $\rmd s^2$ is stationnary for $\cE$ with respect to compactly supported conformal deformations that preserve the total area if and only if there exists a Lagrange multiplier $b\in\R$ such that
$$\int_\Sigma H(\Delta\log|K|+2K)\mu=b\int_\Sigma H\mu$$ for all compactly supported functions $H$, i.e.,
$$\Delta\log|K|=-2K+b.$$
\end{proof}

\begin{rem}
 It is important to emphasize in this result the hypothesis that the curvature does not vanish. Actually, the functional $\cE$ can also be defined on the space $\cM_+(\Sigma)$ (respectively, $\cM_-(\Sigma)$) of smooth metrics on $\Sigma$ with non-negative (respectively, non-positive) curvature. However, item (1) fails on $\cM_+(\Sigma)$. Indeed, assume that $\Sigma$ is compact and orientable and that $\rmd s^2\in\cM_+(\Sigma)$. Considering $\rmd s^2_t=e^{2t}\rmd s^2$ for $t\in\R$, one has $\cE(\rmd s^2_t)=\cE(\rmd s^2)-4\pi t\chi(\Sigma)$. Hence, if $\rmd s^2$ is stationary with respect to conformal deformations of the metric, then $\Sigma$ is a torus. However, there exists generalized Ricci spheres of type $(-2,0,0)$ with non-negative curvature (see Proposition \ref{sphere2}).
\end{rem}

\begin{rem}
A metric in $\cM_*(\Sigma)$ is a gradient Ricci soliton if and only it is stationary for $\cE$ with respect to compactly supported infinitesimally area-preserving deformations of the metric \cite[Theorem 4.1]{bm-imrn}. This is also equivalent to being a 
  generalized Ricci metric of type $(-2,b,0)$ for some $b\in\R$ such that moreover $\mathring\nabla^2\log|K|=0$ \cite[Proposition 2.1]{bm-imrn}. The soliton is shrinking (respectively steady, expanding) if $b>0$ (respectively $b=0$, $b<0$).
\end{rem}

\section{Compact orientable generalized Ricci surfaces} \label{sec:compact}

\subsection{Generalities}

Using an integral formula for functions of absolute value type due to Eschenburg, Guadalupe and Tribuzy \cite{egt}, we get the following lemma for compact orientable generalized Ricci surfaces, which will be useful to obtain some topological obstructions.

\begin{lemma}
Let $(a,b,c)\in\R^3$. Let $(\Sigma,\rmd s^2)$ be a compact orientable generalized Ricci surface of type $(a,b,c)$. If $K$ is not identically $c$, then 
\begin{equation} \label{avtintegral}
  \pi a\chi(\Sigma)+\frac b2\mathrm{Area}(\Sigma)=-2\pi N
 \end{equation}
  where $N$ is the sum of the orders of the zeroes of $\sqrt{|K-c|}$. 
\end{lemma}

\begin{proof}
By Theorem \ref{main}, $\sqrt{|K-c|}$ is of absolute value type. Hence, since this function does not vanish identically, \cite[Lemma 4.1]{egt} gives $$\int_\Sigma\Delta\log\sqrt{|K-c|}\;\mu=-2\pi N.$$ Next, we conclude by Lemma \ref{deltalog} and the Gauss-Bonnet formula.
\end{proof}

We also have another useful integral formula.

\begin{lemma}
 Let $(a,b,c)\in\R^3$. Let $(\Sigma,\rmd s^2)$ be a compact orientable generalized Ricci surface of type $(a,b,c)$. Then 
  \begin{equation} \label{ricciintegral}
  2\int_\Sigma||\nabla K||^2\mu+\int_\Sigma(aK+b)(K-c)^2\mu=0,
 \end{equation}
\end{lemma}

\begin{proof}
 We integrate \eqref{ricci} and use Stokes' formula to get the result.
\end{proof}

\subsection{Spheres}

\begin{prop} \label{sphere}
Let $(a,b,c)\in\R^3$ and let $(\Sigma,\rmd s^2)$ be a generalized Ricci sphere of type $(a,b,c)$.
\begin{enumerate}
\item If $b>0$ and $K$ is not identically $c$, then $a<0$.
 \item If $b=0$ and $K$ is not identically $c$, then $a=-N\in-\N$, where $N$ is the sum of the orders of the zeroes of $\sqrt{|K-c|}$.
 \item If $a=b=0$, then $K$ is constant.
 \item If $b=c=0$, then $K\geqslant0$ and $a\in-2\N$; if moreover $a<0$, then every zero of $\sqrt K$ is of order at most $-a/2$ (in particular, $K$ has to vanish at at least two points).
\end{enumerate}
\end{prop}

\begin{proof}
 \begin{enumerate}
 \item  This is a consequence of \eqref{avtintegral} with $\chi(\Sigma)=2$.
  \item This is a consequence of \eqref{avtintegral} with $\chi(\Sigma)=2$ and $b=0$.
  \item This is a consequence of \eqref{ricciintegral} with $a=b=0$.
  \item By Theorem \ref{main}, $K$ does not change sign, hence $K\geqslant0$ by the Gauss-Bonnet formula, and $K$ is not identically $0$. Then, by (2), $a=-N\in-\N$.
  
  From now on we assume that $N\in\N^*$. The function $\sqrt K$ has isolated zeroes with orders $m_1,\dots,m_n$ with $m_j\in\N^*$ for each $j$ and $\sum_{j=1}^nm_j=N$. Since the metric $K\rmd s^2$ has constant curvature $1-\frac a2>0$ away from the zeroes of $K$, it is homothetic to a spherical metric with conical singularities of angles $2\pi(m_j+1)$, $j\in\{1,\dots,n\}$. Since $m_1,\dots,m_n$ are integers, the developing map of this metric is a rational function $G:\overline\C\to\overline\C$ with critical points of multiplicities $m_1,\dots,m_n$; then the Riemann-Hurwitz formula $$2+\sum_{j=1}^nm_j=2\deg G$$ implies that $N$ is even; also, $m_j\leqslant\deg G-1=N/2$ for each $j$ (and so, in particular, $n\geqslant2$); we refer to \cite[Section 3]{eremenko-gt} for details.
 \end{enumerate}
\end{proof}

\begin{ex} \label{ex:veronese}
Let $c>0$. We recall Section \ref{sec:immersions}.
 \begin{itemize}
 \item A smooth algebraic curve of degree $2$ in $\C\mathbb{P}^2(4c)$ is a compact generalized Ricci sphere of type $(6,-12c,4c)$ with $K\leqslant 4c$. It does not have constant curvature unless it is the image by an isometry of $\C\mathbb{P}^2(4c)$ of the Veronese surface $\{[Z_0,Z_1,Z_2]\in\C\mathbb{P}^2\mid Z_0^2+Z_1^2+Z_2^2=0\}$ \cite[Corollary 3.7]{egt} (see also \cite{nomizu-smyth}).
  \item There exist infinitely many non isometric minimal immersed spheres in $\s^4(c)$ of area $12\pi/c$: see for instance \cite[Propositions 6.15 and 6.21]{barbosa} with $m=k=2$. As noticed in \cite[Introduction]{bryant-sphere}, minimal spheres in $\s^4(c)$ are superminimal. Hence these are generalized Ricci spheres of type $(6,-2c,c)$ with $K\leqslant c$, $K\not\equiv c$, and $N=0$ by \eqref{avtintegral}, so $K<c$.
  \item By Corollary \ref{conformalricci2bis} and Remarks \ref{rkdual} and \ref{homothety}, generalized Ricci spheres of type $(6,-2c,c)$ with $K<c$ are in bijection with generalized Ricci spheres of type $(4,-c,c)$ with $K<c$. Hence there exist infinitely many non isometric such spheres. By \eqref{avtintegral} with $N=0$, their area is $16\pi/c$.
  \item There exist infinitely many non isometric almost complex spheres, hence superminimal, in the standard nearly K\"ahler $\s^6(c)$ of area $24\pi/c$: see \cite[Theorem 5.1]{fernandez} or \cite[Theorem 6]{martins} with $d=6$ (these are immersions because the fact that the area is $24\pi/c$ implies that the total ramification degrees of the indicatrix curve vanish). Hence these are generalized Ricci spheres of type $(6,-c,c)$ with $K\leqslant c$, $K\not\equiv c$, and $N=0$ by \eqref{avtintegral}, so $K<c$.
  \item Reasoning as above, we deduce that there exist infinitely many non isometric generalized Ricci spheres of type $\left(\frac{10}3,-\frac c3,c\right)$ with $K<c$; their area is $40\pi/c$.
 \end{itemize}
\end{ex}

We now construct generalized Ricci spheres of type $(a,0,0)$ satisfying the necessary conditions given in item (4) of Proposition \ref{sphere}. The method of the following construction is inspired by \cite[Proposition 6.2]{moroianu}.

\begin{thm} \label{sphere0}
    Let $\ell\in\N^*$. Let $(m_1,\dots,m_n)$ be a partition of $2\ell$ such that $$m_j\leqslant\ell$$ for every $j\in\{1,\dots,n\}$. Let $p_1,\dots,p_n$ be distinct points of $\overline\C$. Then there exists a conformal generalized Ricci metric of type $(-2\ell,0,0)$ on $\overline\C$ such that $K\geqslant0$, the zeroes of the function $\sqrt K$ are $p_1,\dots,p_n$ and their respective orders are $m_1,\dots,m_n$. 
\end{thm}

\begin{proof} 
  Since  $$2+\sum_{j=1}^{n}m_j=2(\ell+1)$$ and $m_j\leqslant\ell$ for every $j\in\{1,\dots,n\}$, by \cite[Section 3]{eremenko-gt} there exists a metric $\rmd\sigma^2$ of constant curvature $\ell+1>0$ on $\overline\C\setminus\{p_1, p_2,\dots,p_n\}$ with conical singularities at the $p_j$ of respective orders $m_j$ (i.e., angle $2\pi(m_j+1)$), $j\in\{1,\dots,n\}$: this follows from the existence of a rational function $\overline\C\to\overline\C$ of degree $\ell+1$ whose critical points are $p_j$ with multiplicities $m_j$, $j\in\{1,\dots,n\}$, proved by Scherbak \cite{scherbak}.
   
    Around each $p_j$ there is a conformal coordinate $z$ defined on an open set $U_j$ such that $z=0$ at $p_j$ and
    $$\rmd\sigma^2=e^{2v_j}|z|^{2m_j}|\rmd z|^2$$ on $U_j$ with $v_j\in\cC^\infty(U_j,\R)$.
    It is important to notice that, since $m_j\in\N^*$, the function $v_j$ is smooth at the origin.
    
     We set $a=-2\ell$. 
     Let us define a function $\beta:\{p_1, p_2,\dots, p_n\}\to\mathbb{R}$ by
    $$\beta(p_j) = \frac{2m_j +a}{a}.$$
   Hence 
    $$\sum_{j=1}^{n}(\beta(p_j)-1) = \frac{2}{a}\sum_{j=1}^{n}m_j = -2 = -\chi(\overline\C).$$
    So, from \cite[Lemma 6.1]{moroianu}, we obtain a flat metric $\rmd\sigma^2_0$ defined on $\Sigma\setminus\{p_1, p_2,\dots,p_n\}$ that is conformal to $\rmd\sigma^2$ and such that, in each $U_j$, this metric can be  written as 
   $$\rmd\sigma^2_0=e^{2u_j}|z|^{2\beta(p_j)-2}|\rmd z|^2=e^{2u_j}|z|^{\frac{4}{a}m_j}|\rmd z|^2$$ with $z$ as above
    for some $u_j\in\cC^{\infty}(U_j,\mathbb{R})$.

     There is a smooth positive function $V$ on $\overline\C\setminus\{p_1, p_2,\dots,p_n\}$ such that $\rmd\sigma^2 = V\rmd\sigma^2_0$. We may define $\rmd s^2=V^{\frac{2}{2-a}}\rmd\sigma^2_0$ on $\overline\C\setminus\{p_1, p_2,\dots,p_n\}$. Then, by Propositions \ref{propV} and \ref{propVconical2}, $\rmd s^2$ extends to a metric on $\Sigma$ with the desired properties.
\end{proof}

\begin{rem} \label{sphere-reciprocal}
We can give another argument starting with $\rmd\sigma^2$, the same metric with conical singularity as in the previous proof. 
  Let $\Omega$ be the area form of $\rmd\sigma^2$ on $\overline\C\setminus\{p_1, p_2,\dots,p_n\}$. It extends smoothly at the $p_j$ since the $m_j$ are integers.  Applying the Gauss-Bonnet formula with conical singularities \cite[Proposition 1]{troyanov1991prescribing} to $\rmd\sigma^2$ gives
 $$\frac{\ell+1}{2\pi}\mathrm{Area}(\rmd\sigma^2)=\chi(\overline\C)+\sum_{j=1}^nm_j=2(\ell+1)$$ and so $\Area(\rmd\sigma^2)=4\pi=2\pi\chi(\overline\C)$. So, by a theorem of Wallach and Warner \cite{ww}, there exists a conformal smooth metric $\rmd s^2$ on $\overline\C$ whose curvature $2$-form is $\Omega$. We have $\Omega=K\mu$ where $K$ is the curvature of $\rmd s^2$ and $\mu$ its area form. So, $K\geqslant0$ and $\rmd\sigma^2=|K|\rmd s^2$. By Corollary \ref{constantcurvature} (item (2) with $a=-2\ell$ and $c=0$), $\rmd s^2$ is a generalized Ricci metric of type $(-2\ell,0,0)$ on $\Sigma$.
\end{rem}

In the case where the curvature vanishes at only two points, we can make these metrics explicit.

\begin{prop} \label{sphere2} 
For $\ell\in\N^*$ and $\tau\in[0,+\infty)$ we consider on the Riemann sphere $\overline\C$ the metric 
$$\rmd\sigma^2_{\ell,\tau}= \frac{|\rmd z|^2}{\left(|1+\tau z^{\ell+1}|^2+|z|^{2(\ell+1)}\right)^{2/(\ell+1)}}.$$
\begin{enumerate}
 \item For $\ell\in\N^*$ and $\tau\in[0,+\infty)$, the surface $(\overline\C,\rmd\sigma^2_{\ell,\tau})$ is a generalized Ricci sphere of type $(-2\ell,0,0)$ and its curvature vanishes exactly at $0$ and $\infty$.
 \item Let $\ell\in\N^*$. Then any generalized Ricci sphere of type $(-2\ell,0,0)$ whose curvature vanishes at exactly two points is homothetic to $(\overline\C,\rmd\sigma^2_{\ell,\tau})$ for some unique $\tau\in[0,+\infty)$.
 \item Any generalized Ricci sphere of type $(-2,0,0)$ is homothetic to $(\overline\C,\rmd\sigma^2_{1,\tau})$  for some unique $\tau\in[0,+\infty)$.
\end{enumerate}
\end{prop}

\begin{proof} 
\begin{enumerate}
 \item These metrics are obtained by the method of Theorem \ref{sphere0} starting with the expression of spherical metrics with two conical singularities on a sphere by Troyanov \cite{troyanov}; this can be checked in the proof of the next item.
\item Let $\ell\in\N^*$ and let $(\Sigma,\rmd s^2)$ be a generalized Ricci sphere of type $(-2\ell,0,0)$ whose curvature $K$ vanishes at exactly two points. Without loss of generality, we assume that $\Sigma=\overline\C$ and that the zeroes of $K$ are at $0$ and $\infty$. 

As seen in the proof of Proposition \ref{sphere}, $K\geqslant0$ and $K$ is not identically $0$. Then by Corollary \ref{constantcurvature} the metric $K\rmd s^2$ is a metric with constant curvature $\ell+1$ and exactly two conical singularities on a sphere, and the orders of the conical singularities are equal to the orders of $\sqrt K$, hence integers. Moreover, by Proposition \ref{sphere}, the sum of these orders is equal to $2\ell$.

Then by Troyanov's classification of constant curvature metrics with two conical singularities on a sphere \cite{troyanov} (see Theorems I and II and the discussion that follows there), the orders of the conical singularities are equal, and there exists a constant $\tau\geqslant0$ such that, possibly after a conformal change of parameter $z\mapsto\lambda z$, $\lambda\in\C^*$,
 $$K\rmd s^2=\frac{4(\ell+1)|z|^{2\ell}}{\left(|1+\tau z^{\ell+1}|^2+|z|^{2(\ell+1)}\right)^2}|\rmd z|^2.$$
 
 Also, writing $\rmd s^2=e^{-2f}|\rmd z|^2$, then by Lemma \ref{lemmaholo} there exists a holomorphic function $h$ on $\C$ such that $e^{2\ell f}K=|h|^2$. Then, by the aforementioned properties of $\sqrt K$, $h$ does not vanish on $\C^*$, has a zero of order $\ell$ at $0$ and $h(z)z^{-2\ell}\sim\alpha z^{-\ell}$ at $\infty$ for some $\alpha\in\C^*$ (see the discussion in Section \ref{sec:conformal} with the change of variable $z=\psi(w)=1/w$). This implies that $h(z)=\alpha z^\ell$ for all $z\in\C$.
 
 Hence, we have
 $$Ke^{-2f}=\frac{4(\ell+1)|z|^{2\ell}}{\left(|1+\tau z^{\ell+1}|^2+|z|^{2(\ell+1)}\right)^2},\quad\quad Ke^{2\ell f}=|\alpha z^\ell|^2,$$
 so $$\rmd s^2=e^{-2f}|\rmd z|^2=
 \frac{(4(\ell+1))^{1/(\ell+1)}|\rmd z|^2}{|\alpha|^{2/(\ell+1)}\left(|1+\tau z^{\ell+1}|^2+|z|^{2(\ell+1)}\right)^{2/(\ell+1)}},$$ which proves the existence of $\tau$ as claimed.
 
 Finally, two homothetic generalized Ricci surfaces of type $(-2\ell,0,0)$ have the same associated constant curvature metric $K\rmd s^2$ (up to an isometry), and $\tau$ is a function of the distance between the two conical singularities in this metric \cite{troyanov}. This proves the uniqueness of $\tau$.
 \item Let $(\Sigma,\rmd s^2)$ be a generalized Ricci sphere of type $(-2,0,0)$. By Proposition \ref{sphere}, $K\geqslant0$ and $\sqrt K$ has exactly two zeroes, each of order $1$. Then the result follows from (2) with $\ell=1$.
 \end{enumerate}
\end{proof}

\begin{rem} \label{spherenotsmooth}
 We may also consider the same metric with $\tau=0$ and replacing the integer $\ell\in\N^*$ by a real number $\beta>0$. Then, the metric is at least of class $\cC^2$ on $\overline\C$, smooth on $\C^*$, its curvature vanishes at $0$ and $\infty$, and it is a generalized Ricci metric of type $(-2\beta,0,0)$ on $\C^*$. However, if $\beta\notin\N^*$, then it is not smooth at $0$ and $\infty$; see also Remark \ref{rknotsmooth}.
\end{rem}

We now construct generalized Ricci spheres of type $(a,0,c)$ with $a\in-2\N^*$ and $c\in\R^*$ admitting a rotational invariance.

 \begin{prop} \label{spherea0c}
 Let $\ell\in\N^*$, $c\in\R^*$ and $\varepsilon\in\{-1,1\}$. If $c<0$, we assume that $\varepsilon=1$. Then there exist
rotationally invariant generalized Ricci metrics of type $(-2\ell,0,c)$ on a sphere such that $\sgn(K-c)=\varepsilon$.

Moreover, there exist infinitely many such metrics on a sphere that are not isometric one to another.
 \end{prop}
 
 \begin{proof}
  Let $\xi\in\R^*$ such that $\sgn\xi=\varepsilon$. If $c>0$ and $\xi<0$, we moreover assume that $\xi+\left(\frac{\ell}{\ell+1}\right)^\ell c^{\ell+1}>0$. Let $y_0\in\R$. Then there exists a smooth solution $y:(-\delta,\delta)\to\R$ to the ordinary differential equation
  \begin{equation} \label{sphere-edoy}
   y'(t)=\frac{cte^{-2y(t)}+\frac\xi{\ell+1}t^{2\ell+1}e^{-2(\ell+1)y(t)}}{1+\sqrt{1-ct^2e^{-2y(t)}-\frac\xi{\ell+1}t^{2\ell+2}e^{-2(\ell+1)y(t)}}}
  \end{equation}
with the initial condition $y(0)=y_0$ for some $\delta>0$, and this function $y$ is even.

For $u\in(-\infty,\log\delta)$ we set 
$$L(u)=y(e^u)-u.$$ Then, on $(-\infty,\log\delta)$, this function $L$ satisfies
\begin{equation} \label{sphere-edoL1}
 L'(u)^2+\Phi(L(u))=1
\end{equation}
with $\Phi(r)=ce^{-2r}+\frac\xi{\ell+1}e^{-2(\ell+1)r}$ and
\begin{equation} \label{sphere-edoL2}
 L''(u)=ce^{-2L(u)}+\xi e^{-2(\ell+1)L(u)}.
\end{equation}
Moreover, we have $\displaystyle{\lim_{u\to-\infty}L(u)=+\infty}$.

The hypotheses on $c$ and $\xi$ imply that there is no constant solution to \eqref{sphere-edoL1} and  \eqref{sphere-edoL2}. Also,
\begin{itemize}
 \item if $c>0$ and $\xi>0$, then $\Phi$ is decreasing on $\R$, with $\displaystyle{\lim_{r\to-\infty}\Phi(r)=+\infty}$ and $\displaystyle{\lim_{r\to+\infty}\Phi(r)=0}$;
 \item if $c>0$ and $\xi<0$, then $\Phi$ is increasing on $\left(-\infty,-\frac1{2\ell}\log\left(-\frac c\xi\right)\right)$ and decreasing on $\left(-\frac1{2\ell}\log\left(-\frac c\xi\right),+\infty\right)$, with $\displaystyle{\lim_{r\to-\infty}\Phi(r)=-\infty}$, $\displaystyle{\lim_{r\to+\infty}\Phi(r)=0}$ and $\max\Phi=\frac{\ell c}{\ell+1}\left(-\frac c\xi\right)^{-\frac1\ell}>1$;
 \item if $c<0$ and $\xi>0$, then $\Phi$ is decreasing on $\left(-\infty,-\frac1{2\ell}\log\left(-\frac c\xi\right)\right)$ and increasing on $\left(-\frac1{2\ell}\log\left(-\frac c\xi\right),+\infty\right)$, with $\displaystyle{\lim_{r\to-\infty}\Phi(r)=+\infty}$, $\displaystyle{\lim_{r\to+\infty}\Phi(r)=0}$ and $\min\Phi=\frac{\ell c}{\ell+1}\left(-\frac c\xi\right)^{-\frac1\ell}<0$.
\end{itemize}
Hence, in all cases we can extend $L$ to a smooth function $L:\R\to\R$ satisfying \eqref{sphere-edoL1} and \eqref{sphere-edoL2} on $\R$ and for which there exists $q\in\R$ such that
\begin{equation} \label{sphere-Leven}
 L(-u)=L(u+2q)
\end{equation}
for all $u\in\R$. Consequently, we can extend $y$ to an even smooth function $y:\R\to\R$ setting $$y(t)=L(\log|t|)+\log|t|$$ for all $t\in\R^*$. Then we deduce from equation \eqref{sphere-edoL2} that 
\begin{equation} \label{sphere-edoy2}
 y''(t)+\frac1ty'(t)=ce^{-2y(t)}+\xi t^{2\ell}e^{-2(\ell+1)y(t)}
\end{equation}
for $t\in\R^*$.

We now define $f:\C\to\R$ by $f(z)=y(|z|)$ and we consider the metric $\rmd s^2=e^{-2f}|\rmd z|^2$ on $\C$. Then $\rmd s^2$ is smooth on $\C$ and we deduce from \eqref{sphere-edoy2} that its curvature $K=4e^{2f}f_{z\bar z}$ satisfies $e^{2\ell f}(K-c)=\varepsilon|h|^2$ with $h(z)=\sqrt{|\xi|}z^\ell$. Then, by Lemma \ref{lemmaholo}, $\rmd s^2$ is a generalized Ricci metric of type $(-2\ell,0,c)$ on $\C$ with $\sgn(K-c)=\varepsilon$.

It remains to check that $\rmd s^2$ extends smoothly at $\infty$. We first deduce from \eqref{sphere-Leven} that $$y\left(\frac{e^{2q}}t\right)+2\log t=y(t)+2q$$ for all $t>0$.
By applying the change of variable $w=\frac{e^{2q}}z$ we have
$$\rmd s^2=e^{-2\left(y\left(\frac{e^{2q}}{|w|}\right)+2\log|w|-2q\right)}|\rmd w|^2
=e^{-2y(|w|)}|\rmd w|^2,$$ which extends smoothly at $w=0$. This proves that $\rmd s^2$ has the required properties.

(We notice that, if we change the initial condition $y_0$, then $y$ is replaced by $t\mapsto y(e^\beta t)-\beta$ for some $\beta\in\R$, but then the new metric is isometric to the initial one via $z\mapsto e^\beta z$.)

Let $\rmd\tilde s^2=e^{-2\tilde y(|z|)}|\rmd z|^2$ be another metric on $\overline\C$ with $\tilde y$ obtained in the same way replacing $\xi$ by $\tilde\xi\in\R^*$ with the same sign. Assume that $\rmd s$ and $\rmd\tilde s^2$ are isometric. An isometry between them is a conformal or anticonformal automorphism of $\overline\C$ fixing the set $\{0,\infty\}$. Since $z\mapsto e^{i\theta}z$ for $\theta\in\R$, $z\mapsto\bar z$ and $z\mapsto\frac{e^{2q}}z$ are isometries of $\rmd s^2$, we may assume that this isometry is
 $z\mapsto e^\beta z$ for some $\beta\in\R$.
Then $\tilde y(t)=y(e^\beta t)-\beta$ for all $t\in\R$, and so $\tilde y$ satisfies \eqref{sphere-edoy2}. Since it also satisfies \eqref{sphere-edoy2} with $\xi$ replaced by $\tilde\xi$, we conclude that $\tilde\xi=\xi$.

Consequently, if we vary the value of $\xi$, then we obtain metrics that are not isometric one to another.
 \end{proof}

 \begin{rem}
  For $\ell=1$ we have an explicit expression: $y(t)=\frac12\log\left(\left(t^2+\frac c4\right)^2+\frac\xi8\right)$, with $\xi>0$ or $-\frac{c^2}2<\xi<0$ if $c>0$, and $\xi>0$ if $c<0$.
 \end{rem}

\subsection{Tori}

\begin{prop} \label{torus}
Let $(a,b,c)\in\R^3$ and let $(\Sigma,\rmd s^2)$ be a generalized Ricci torus of type $(a,b,c)$.
 \begin{enumerate}
  \item If $c=0$, then $\Sigma$ is flat.
  \item If $c>0$, then $b\leqslant0$ and $K\leqslant c$; if moreover $b=0$ and $\Sigma$ is not flat, then $a>0$ and $K<c$.
  \item If $c<0$, then $b\leqslant0$ and $K\geqslant c$; if moreover $b=0$ and $\Sigma$ is not flat, then $a<0$ and $K>c$.
 \end{enumerate}
\end{prop}

\begin{proof}
 \begin{enumerate}
  \item If $c=0$, then Theorem \ref{main} implies that $K$ has constant sign. We conclude that $K\equiv0$ by the Gauss-Bonnet formula.
  \item If $c>0$, then Theorem \ref{main} implies that $K-c$ has constant sign, and so $K\leqslant c$ by the Gauss-Bonnet formula, and $K$ is not identically $c$. Then, \eqref{avtintegral} with $\chi(\Sigma)=0$ implies that $b\leqslant0$.
  
  Assume moreover that $b=0$ and $\Sigma$ is not flat. Then $N=0$, i.e., $K<c$. Moreover, $K$ is not constant (otherwise the Gauss-Bonnet formula would give $K\equiv0$). Considering \eqref{ricci} at a point where $K$ attains a minimum and at a point where $K$ attains a maximum yields $a\min K\leqslant0\leqslant a\max K$. Consequently, $a\geqslant0$. Finally, if $a=0$, then \eqref{ricciintegral} implies that $K$ is constant, which is a contradiction.
  \item If $c<0$, we proceed with the same arguments as when $c>0$.
 \end{enumerate}
\end{proof}

\begin{ex} \label{ex:torus}
Let $c>0$.
There exist non flat generalized Ricci tori of type $(6,-12c,4c)$, $(6,-2c,c)$ and $(6,-c,c)$; see Example \ref{ex:superminimal}.
\end{ex}

\begin{ex} \label{delaunay} {\it Delaunay-type metrics.}
Let $(a,c)\in\R^2$ such that $a\cdot c>0$. The ordinary differential equation 
\begin{equation} \label{edodelaunay}
 y''=-ce^{(a-2)y}+ce^{-2y}
\end{equation}
admits a prime integral:
$$\frac12(y')^2+\Phi(y)=E$$
with $\Phi(r)=\frac c{a-2}e^{(a-2)r}+\frac c2e^{-2r}$ if $a\neq2$, $\Phi(r)=cr+\frac c2e^{-2r}$ if $a=2$.
Then the function $\Phi$ attains a minimum at $0$, is decreasing on $(-\infty,0)$, increasing on $(0,+\infty)$ and satisfies $\displaystyle{\lim_{r\to-\infty}\Phi(r)=+\infty}$. If we set $\displaystyle{E\in\left(\min\Phi,\lim_{r\to+\infty}\Phi(r)\right)}$, then the level set of the prime integral is a compact smooth curve around the equilibrium $(0,0)$. Hence this yields a non-constant solution $y$ of \eqref{edodelaunay} that is periodic for some period $T>0$. Hence, if we set $f(u+iv)=y(v)$, then $f_{z\bar z}=-\frac c4e^{(a-2)f}+\frac c4e^{-2f}$, i.e., the curvature $K$ of the metric $e^{-2f}|\rmd z|^2$ satisfies $e^{-af}(K-c)=-c$. Hence, by Lemma \ref{lemmaholo} (with $h\equiv\sqrt{|c|}$, $b=0$ and $\varepsilon=\sgn c$), $e^{-2f}|\rmd z|^2$ is a generalized Ricci metric on $\C$ of type $(a,0,c)$. It passes to the torus $\C/\Gamma$ with $\Gamma=\alpha\Z\oplus(\beta+iT)\Z$ for any $(\alpha,\beta)\in\R^*\times\R$. This metric cannot be flat, since $f$ is not harmonic.

In conclusion, in each conformal class on a torus there exist non constant curvature generalized Ricci metrics of type $(a,0,c)$. (If $a=4$ and $c>0$, these are the metrics of Delaunay constant mean curvature surfaces in $\R^3$.)
\end{ex}

\subsection{Surfaces of genus at least 2}

\begin{prop} \label{highgenus}
Let $(a,b,c)\in\R^3$ and let $(\Sigma,\rmd s^2)$ be a compact orientable generalized Ricci surface of type $(a,b,c)$ of genus $g\geqslant2$.
\begin{enumerate}
\item If $b>0$ and $K$ is not identically $c$, then $a>0$.
 \item If $b=0$ and $K$ is not identically $c$, then $a=\frac N{g-1}\in\frac1{g-1}\N$, where $N$ is the sum of the orders of the zeroes of $\sqrt{|K-c|}$.
 \item If $a=b=0$, then $K$ is constant.
\end{enumerate}
\end{prop}

\begin{proof}
 \begin{enumerate}
 \item This is a consequence of \eqref{avtintegral} with $\chi(\Sigma)=2(1-g)<0$.
  \item This is a consequence of \eqref{avtintegral} with $\chi(\Sigma)=2(1-g)$ and $b=0$. 
  \item This is a consequence of \eqref{ricciintegral} with $a=b=0$.
\end{enumerate}
\end{proof}

We first give some examples from the existing literature (recall Section \ref{sec:immersions}).

\begin{ex} \label{ex:superminimal}
Let $c>0$.
The following examples do not have constant curvature when they have positive genus (by Remark \ref{kappa} and the Gauss-Bonnet formula). 
\begin{itemize}
\item  A smooth algebraic curve of degree $d\in\N^*$ in $\C\mathbb{P}^2(4c)$ is a compact orientable generalized Ricci surface of type $(6,-12c,4c)$ of genus $(d-1)(d-2)/2$.
 \item Let $\Sigma$ be a compact Riemann surface. 
Bryant \cite[Corollary H]{bryant-sphere} proved that there exists a conformal superminimal immersion $X:\Sigma\to\s^4(c)$. Hence, the metric induced by $X$ on $\Sigma$ is a conformal generalized Ricci metric of type  $(6,-2c,c)$. 
\item Let $\Sigma$ be a compact Riemann surface. It is pointed in \cite{madnick} that
Rowland \cite{rowland} proved that there exists a conformal superminimal pseudholomorphic embedding $X:\Sigma\to\s^6(c)$. Hence, the metric induced by $X$ on $\Sigma$ is a conformal generalized Ricci metric of type  $(6,-c,c)$.
\end{itemize}
\end{ex}

\begin{ex} \label{ex:superminimal2}
Let $c<0$. Let $\Sigma$ be a compact Riemann surface of genus at least $2$.
\begin{itemize}
 \item There exist non totally geodesic holomorphic immersions $X:\Sigma\to\C\h^2(4c)$ (see \cite{mcintosh23}, where they are parametrized by the sets $\mathcal{W}_\tau$).
 Hence, the metric induced by $X$ on $\Sigma$ is a conformal generalized Ricci metric of type  $(6,-12c,4c)$. It does not have constant curvature by Remark \ref{kappa} and the fact that $K\leqslant4c$.
 \item Loftin and McIntosh \cite[Theorem 4.6]{loftinmcintosh19} proved that there exist a non totally geodesic conformal superminimal immersions $X:\Sigma\to\h^4(c)$ (also, recently, Bronstein \cite{bronstein} proved that there exist embeddings provided the genus is large enough).
 Hence, the metric induced by $X$ on $\Sigma$ is a conformal generalized Ricci metric of type  $(6,-2c,c)$. It does not have constant curvature by \cite[Theorem 1]{kenmotsu} or by Remark \ref{kappa} and the fact that $K\leqslant c$.
\end{itemize}
\end{ex}

\begin{ex} \label{ex:lawson}
 Let $c>0$. As in Example \ref{ex:superminimal}, the following examples do not have constant curvature.
 \begin{itemize}
  \item Lawson \cite[Theorem 2]{lawson} constructed compact orientable embedded minimal surfaces in the sphere $\s^3(c)$ of any genus $g\geqslant2$; these are generalized Ricci surfaces of type $(4,0,c)$. 
  \item Haskins and Kapouleas \cite{hk} constructed (by gluing techniques) compact orientable special Legendrian surfaces in the sphere $\s^5(c)$ of any odd genus $g\geqslant3$ and of genus $4$. The image of such a surface by the Hopf fibration $\s^5(c)\to\C\mathbb{P}^2(4c)$ is a Lagrangian minimal surface, hence a generalized Ricci surface of type $(6,0,c)$.
 \end{itemize}

\end{ex}

We now construct generalized Ricci surfaces of type $(a,0,0)$ for $a$ satisfying the necessary condition in Proposition \ref{highgenus}. A method (analogue to that of Remark \ref{sphere-reciprocal}) consists in using the next lemma. In the case where $a\neq2$, we will however follow a method similar to that of Theorem \ref{sphere0} and inspired by \cite[Proposition 6.2]{moroianu}, based on the two metrics with conical singularities arising from Corollary \ref{constantcurvature}; this is more explicit and we think it is interesting on its own.

\begin{lemma} \label{reciprocal}
Let $g\geqslant2$ be an integer. Let $a>0$.
 Let $\Sigma$ be a compact orientable surface of genus $g$. Let $\{p_1, p_2,\dots, p_n\}$ be a set of distinct points of $\Sigma$.
 
 Assume that there exists a metric  $\rmd\sigma^2$ of constant curvature $\frac a2-1$ on $\Sigma\setminus\{p_1, p_2,\dots,p_n\}$ with conical singularities at the $p_j$ of respective orders $m_j\in\N^*$, $j\in\{1,\dots,n\}$, and such that the total area of $\rmd\sigma^2$ is equal to $4\pi(g-1)$.
    Then there exists a generalized Ricci metric $\rmd s^2$ of type $(a,0,0)$ on $\Sigma$ such that $K\leqslant0$ and $\rmd\sigma^2=|K|\rmd s^2$.
    
    Moreover, this metric is unique up to a homothety.
\end{lemma}

\begin{proof}
  Let $\Omega$ be the area form of $\rmd\sigma^2$ on $\Sigma\setminus\{p_1, p_2,\dots,p_n\}$. It extends smoothly at the $p_j$ since the $m_j$ are integers. 
We have $\int_\Sigma\Omega=4\pi(g-1)=-2\pi\chi(\Sigma)$, so, by a theorem of Wallach and Warner \cite{ww}, there exists a conformal smooth metric $\rmd s^2$ on $\Sigma$ whose curvature $2$-form is $-\Omega$.
 
Then we have $-\Omega=K\mu$ where $K$ is the curvature of $\rmd s^2$ and $\mu$ its area form. So, $K\leqslant0$ and $\rmd\sigma^2=|K|\rmd s^2$. By Corollary \ref{constantcurvature} (item (2) with $c=0$), $\rmd s^2$ is a generalized Ricci metric of type $(a,0,0)$ on $\Sigma$.

Finally, if $e^{-2F}\rmd s^2$ with $F:\Sigma\to\R$ smooth is a metric with the same curvature form as $\rmd s^2$, then $F$ is harmonic, hence constant. This proves the uniqueness of $\rmd s^2$ up to a homothety.
\end{proof}

\begin{thm} \label{a00highgenus}
    Let $g\geqslant2$ be an integer. Let $a>0$ such that $(g-1)a\in\N^*$. Let $(m_1,\dots,m_n)$ be a partition of $(g-1)a$. 
    \begin{enumerate}
     \item There exists a compact orientable generalized Ricci surface of type $(a, 0, 0)$ of genus $g$ such that $K\leqslant0$ and such that the function $\sqrt{|K|}$ has exactly $n$ zeroes and their respective orders are $m_1,\dots,m_n$.
     \item Assume moreover that $(g-1)a<2g$ or that $(g-1)a$ is odd.
     
     Let $\Sigma$ be a compact Riemann surface of genus $g$ and let $p_1,\dots,p_n$ be distinct points of $\Sigma$. Then there exists a conformal generalized Ricci metric of type $(a, 0, 0)$ on $\Sigma$ such that $K\leqslant0$, the zeroes of the function $\sqrt{|K|}$ are $p_1,\dots,p_n$ and their respective orders are $m_1,\dots,m_n$.
     
     Moreover, if $a\in(0,2]$, then two such metrics are homothetic.
    \end{enumerate}
\end{thm}

\begin{proof}
\begin{enumerate}
 \item  If $a\leqslant2$ then (1) follows from (2), so we will only prove (1) for $a>2$.
 Then we have 
     $$2-2g + \sum_{j=1}^{n}m_j=(a-2)(g-1)>0.$$
    Hence, we know from  \cite[Theorem A]{mondello2019spherical} that there exist a compact oriented Riemann surface $\Sigma$ of genus $g$, a set of distinct points $\{p_1, p_2,\dots, p_n\}$ and a conformal metric $\rmd\sigma^2$ of constant curvature $\frac a2-1>0$ on $\Sigma\setminus\{p_1, p_2,\dots,p_n\}$ with conical singularities at the $p_j$ of respective orders $m_j$ (i.e., angle $2\pi(m_j+1)$), $j\in\{1,\dots,n\}$.
    
    Around each $p_j$ there is a conformal coordinate $z$ defined on an open set $U_j$ such that $z=0$ at $p_j$ and
    $$\rmd\sigma^2=e^{2v_j}|z|^{2m_j}|\rmd z|^2$$ on $U_j$ with $v_j\in\cC^\infty(U_j,\R)$.
    It is important to notice that, since $m_j\in\N^*$, the function $v_j$ is smooth at the origin.

     Let us define a function $\beta:\{p_1, p_2,\dots, p_n\}\rightarrow\mathbb{R}$ by
    $$\beta(p_j) = \frac{2m_j +a}a.$$
   Hence 
    $$\sum_{j=1}^{n}(\beta(p_j)-1) = \frac{2}{a}\sum_{j=1}^{n}m_j = 2(g-1) = -\chi(\Sigma).$$
    So, from \cite[Lemma 6.1]{moroianu}, we obtain a flat metric $\rmd\sigma^2_0$ defined on $\Sigma\setminus\{p_1, p_2,\dots,p_n\}$ that is conformal to $\rmd\sigma^2$ and such that, in each $U_j$, this metric can be  written as 
   $$\rmd\sigma^2_0=e^{2u_j}|z|^{2\beta(p_j)-2}|\rmd z|^2=e^{2u_j}|z|^{\frac{4}{a}m_j}|\rmd z|^2$$ with $z$ as above
    for some $u_j\in\cC^{\infty}(U_j,\mathbb{R})$.

     There is a smooth positive function $V$ on $\Sigma\setminus\{p_1, p_2,\dots,p_n\}$ such that $\rmd\sigma^2 = V\rmd\sigma^2_0$. We may define $\rmd s^2=V^{\frac{2}{2-a}}\rmd\sigma^2_0$ on $\Sigma\setminus\{p_1, p_2,\dots,p_n\}$. Then, by Propositions \ref{propV} and \ref{propVconical}, $\rmd s^2$ extends to a metric on $\Sigma$ with the desired properties. This concludes the proof of (1) when $a>2$.
    \item  We first treat the case where $a\neq2$. We claim that there exists a conformal metric $\rmd\sigma^2$ of constant curvature $\frac a2-1$ on $\Sigma\setminus\{p_1, p_2,\dots,p_n\}$ with conical singularities at the $p_j$ of respective orders $m_j$, $j\in\{1,\dots,n\}$ on $\Sigma$. 
    \begin{itemize}
     \item If $a\in(0,2)$, then $\frac a2-1<0$ and $$2-2g + \sum_{j=1}^{n}m_j=(a-2)(g-1)<0,$$ so this follows from \cite{mcowen} or \cite[Theorem A]{troyanov1991prescribing}.
      \item If $a>2$ and $(g-1)a<2g$, then $\frac a2-1>0$ and $$2-2g + \sum_{j=1}^{n}m_j=(a-2)(g-1)\in(0,2),$$ so this follows from \cite[Theorem C]{troyanov1991prescribing}.
        \item If $a>2$, $(g-1)a>2g$ and $(g-1)a$ is odd, then $\frac a2-1>0$, $$2-2g + \sum_{j=1}^{n}m_j=(a-2)(g-1)>2\textrm{  and }\notin2\N,$$ so this follows from \cite[Theorem 1.1]{bdmm}.
    \end{itemize}
This proves the claim. Then, we conclude the proof of existence of the desired generalized Ricci metric as in (1).
        
    We now prove uniqueness up to homotheties in the case where $a\in(0,2)$. If $\rmd s^2$ is a generalized Ricci metric with the required properties, then by Corollary \ref{constantcurvature} the metric $|K|\rmd s^2$ has the same properties as $\rmd\sigma^2$, hence is equal to $\rmd\sigma^2$ by uniqueness in \cite{mcowen} or \cite[Theorem A]{troyanov1991prescribing}. Also, the metric $|K|^{2/a}\rmd s^2$ is flat on $\Sigma\setminus\{p_1, p_2,\dots,p_n\}$ with conical singularities at the $p_j$ of respective orders $2m_j/a$, hence it is homothetic to $\rmd\sigma_0^2$ by \cite[p. 90]{troyanov-flat}. This concludes the proof in the case where $a\in(0,2)$.
    
    We now consider the case where $a=2$. We have $$2-2g + \sum_{j=1}^{n}m_j=0,$$ so, by \cite[p. 90]{troyanov-flat} or \cite[Lemma 6.1]{moroianu}, there exists  a conformal flat metric $\rmd\sigma^2_0$ defined on $\Sigma\setminus\{p_1, p_2,\dots,p_n\}$ with conical singularities at the $p_j$ of respective orders  $m_j$, $j\in\{1,\dots,n\}$. Mutliplying this metric by a positive constant if necessary, we may assume its total area is equal to $4\pi(g-1)$. Hence, the existence of $\rmd s^2$ follows from Lemma \ref{reciprocal}.
     
 To prove the uniqueness of $\rmd s^2$ up to homotheties, we first notice that, as in the case where $a\in(0,2)$, the metric $|K|\rmd s^2$ is homothetic to $\rmd\sigma_0^2$. Then, the Gauss-Bonnet formula implies $\int_\Sigma|K|\rmd s^2=4\pi(g-1)$, so $|K|\rmd s^2=\rmd\sigma_0^2$.
  Consequently, $\rmd s^2$ is necessarily obtained by the above construction. But this construction provides a unique metric up to homotheties by Lemma \ref{reciprocal}. This concludes the proof in the case where $a=2$.
    \end{enumerate}
\end{proof}

\begin{rem}
 We can also use Lemma \ref{reciprocal} to prove (1) and (2) for any $a>0$. Let $\Sigma$ and $\rmd\sigma^2$ be as in the beginning of the proof of (1) or (2). Applying the Gauss-Bonnet formula with conical singularities \cite[Proposition 1]{troyanov1991prescribing} to $\rmd\sigma^2$ gives
 $$\frac1{2\pi}\left(\frac a2-1\right)\mathrm{Area}(\rmd\sigma^2)=\chi(\Sigma)+\sum_{j=1}^nm_j=(a-2)(g-1)$$ and so $\Area(\rmd\sigma^2)=4\pi(g-1)$ if $a\neq2$. Then existence follows from Lemma \ref{reciprocal}, and the uniqueness up to homotheties when $a\in(0,2)$ follows from the uniqueness of $\rmd\sigma^2$ and Lemma \ref{reciprocal}.
\end{rem}

\begin{rem}
 In the remaining case, i.e., $(g-1)a\geqslant2g$ and even, we do not know if there exists, on a given compact Riemann surface $\Sigma$ of genus $g$, a conformal generalized Ricci metric of type $(a,0,0)$, even without prescribing properties of the zeroes of $K$. This is equivalent to the existence of a conformal positive constant curvature metric with conical singularities on $\Sigma$ whose orders $m_1,\dots,m_n$ are a partition of $(g-1)a$. See \cite{mondello2019spherical} for a general discussion. A sufficient condition is the existence of a holomorphic branched cover $G:\Sigma\to\overline\C$ of degree $d=\frac12(g-1)(a-2)\in\N^*$; we can consider the pullback by $G$ of a conformal spherical metric on $\overline\C$; indeed, the Riemann-Hurwitz formula reads $2-2g+\sum_{j=1}^nm_j=2d$. For instance, if $\Sigma$ is hyperelliptic and $(g-1)(a-2)\in4\N^*$, one may take $G=\Phi^{d/2}$ where $\Phi:\Sigma\to\overline\C$ is a holomorphic branched cover of degree $2$. When $a=4$ this is \cite[Corollary 6.3 and Example 6.4]{moroianu}.
\end{rem}

\begin{rem}
 It follows from item (2) of Theorem \ref{a00highgenus} that every compact Riemann surface of genus at least $2$ admits conformal Ricci metrics of type $(1,0,0)$ and of type $(2,0,0)$.
\end{rem}

We will now construct generalized Ricci surfaces of type $(a,0,c)$ of genus $g\geqslant2$ for $a>0$ and $c<0$ satisfying the necessary condition given in Proposition \ref{highgenus}. The strategy is to look for such a metric of the form $e^{-2\Psi}\rmd\sigma^2$ where $\rmd\sigma^2$ is a metric of constant curvature $c$, and to study the partial differential equation satisfied by the function $\Psi$. We first need a lemma.

 \begin{lemma} \label{pde}
 Let $g\geqslant2$ be an integer and let $\Sigma$ be a compact Riemann surface of genus $g$. Let $a>0$, $c<0$ and $\rmd\sigma^2$ be the conformal metric of constant curvature $c$ on $\Sigma$. Let $\underline\Delta$ denote the Laplace-Beltrami operator of $\rmd\sigma^2$.
 \begin{enumerate}
  \item For any smooth function $\Theta:\Sigma\to\R_+$, there exists a smooth function $\Psi:\Sigma\to\R$ such that
 $$\underline\Delta\Psi+c-ce^{-2\Psi}-\Theta e^{(a-2)\Psi}=0.$$ 
 \item There exists a constant $R>0$ such that, for any smooth function $\Theta:\Sigma\to[0,R]$, there exists a smooth function $\Psi:\Sigma\to\R$ such that
 $$\underline\Delta\Psi+c-ce^{-2\Psi}+\Theta e^{(a-2)\Psi}=0.$$ 
 \end{enumerate}
 \end{lemma}

 \begin{proof}
 We use the method of subsolutions and supersolutions, as in \cite[Section 3.1]{loftinmcintosh} and \cite[proof of Theorem 5.1]{lmi-dedicata}. Let $\varepsilon\in\{-1,1\}$ and $\Theta:\Sigma\to\R_+$ be a smooth function. For a function $\Psi:\Sigma\to\R$ of class $\cC^2$ we set
 $$\cL(\Psi)=\underline\Delta\Psi+c-ce^{-2\Psi}-\varepsilon\Theta e^{(a-2)\Psi}.$$ If $\Theta\equiv0$ then $\Psi\equiv0$ is a desired solution. So we now assume that $\Theta$ is not identically zero. We set $M=\max_\Sigma\Theta>0$ and we consider the function $$P(X)=\varepsilon MX^{a/2}-cX+c.$$ 
 
 We first deal with the case where $\varepsilon=1$. We have $\cL(0)\leqslant0$. On the other hand, we have $\cL(\Psi)\geqslant\underline\Delta\Psi-e^{-2\Psi}P(e^{2\Psi})$. Since $P(0)=c<0$ and $P(1)=M>0$, there is $t\in(0,1)$ such that $P(t)=0$, and so 
  $\cL(\frac12\log t)\geqslant 0$. Since $\frac12\log t<0$, there is a smooth solution $\Psi:\Sigma\to[\frac12\log t,0]$ to $\cL(\Psi)=0$. This proves (1).
 
 We now consider the case where  $\varepsilon=-1$. We have $\cL(0)\geqslant0$. On the other hand, we have $\cL(\Psi)\leqslant\underline\Delta\Psi-e^{-2\Psi}P(e^{2\Psi})$. We fix $t>1$. Then $-ct+c>0$ so there exists $R>0$ such that $-Rt^{a/2}-ct+c>0$. We now assume that $M\leqslant R$. Then $P(t)>0$, so we have $\cL(\frac12\log t)\leqslant 0$. Since $\frac12\log t>0$, there is a smooth solution $\Psi:\Sigma\to[0,\frac12\log t]$ to $\cL(\Psi)=0$. This proves (2).
 \end{proof}

 \begin{thm} \label{a0chighgenus}
 Let $g\geqslant2$ be an integer. Let $a>0$ such that $(g-1)a\in\N^*$. Let $c<0$ and $\varepsilon\in\{-1,1\}$.  Let $\Sigma$ be a compact Riemann surface of genus $g$ and let $p_1,\dots,p_n$ be distinct points of $\Sigma$. Let $(m_1,\dots,m_n)$ be a partition of $(g-1)a$.  Then there exists a conformal generalized Ricci metric of type $(a,0,c)$ on $\Sigma$ such that $\sgn(K-c)=\varepsilon$, the zeroes of the function $\sqrt{|K-c|}$ are $p_1,\dots,p_n$ and their respective orders are $m_1,\dots,m_n$.
 
 Moreover, there exist infinitely many such metrics on $\Sigma$ that are not isometric one to another.
 \end{thm}

 \begin{proof}
 Let $\rmd\sigma^2$ be the conformal metric of constant curvature $c$ on $\Sigma$. Around each $p_j$ we consider a conformal coordinate $z$ defined on an open set $U_j$ such that $z=0$ at $p_j$. As in the proof of Theorem \ref{a00highgenus}, from \cite[Lemma 6.1]{moroianu} we obtain a flat metric $\rmd\sigma^2_0$ defined on $\Sigma\setminus\{p_1, p_2,\dots,p_n\}$ that is conformal to $\rmd\sigma^2$ and such that, in each $U_j$, this metric can be  written as 
   $\rmd\sigma^2_0=e^{2u_j}|z|^{\frac4am_j}|\rmd z|^2$ 
    for some $u_j\in\cC^{\infty}(U_j,\mathbb{R})$.
 Then we may write $\rmd\sigma_0^2=\Theta^{2/a}\rmd\sigma^2$ with $\Theta:\Sigma\to\R_+$ smooth and vanishing exactly at the $p_j$.
 
  Let $\underline\Delta$ be the Laplace-Beltrami operator of $\rmd\sigma^2$. By Lemma \ref{pde}, multiplying the metric $\rmd\sigma_0^2$ by a positive constant if necessary in the case where $\varepsilon=-1$, there exists a smooth function $\Psi:\Sigma\to\R$ such that $$\underline\Delta\Psi+c-ce^{-2\Psi}-\varepsilon\Theta e^{(a-2)\Psi}=0.$$
  
 We now claim that the metric $\rmd s^2=e^{-2\Psi}\rmd\sigma^2$ has the required properties. If $z$ is a local conformal coordinate on an open simply connected set $U\subset\Sigma$, we write $\rmd s^2=e^{-2f}|\rmd z|^2$ and $\rmd\sigma^2=e^{-2V}|\rmd z|^2$. Then $\rmd\sigma_0^2=\Theta^{2/a}e^{-2V}|\rmd z|^2$. Since this metric is flat (away from its singular points), $\Theta^{2/a}e^{-2V}$ is log-harmonic and so there exists a holomorphic map $h:U\to\C$ such that $\Theta e^{-aV}=|h|^2$.
  Also, $\Psi=f-V$ and, since $\rmd\sigma^2$ has constant curvature $c$, we compute that the curvature of $\rmd s^2$ is
 $$K=e^{2\Psi}(c+\underline\Delta\Psi)=c+\varepsilon\Theta e^{a\Psi}.$$ On $U$ this reads $$K=c+\varepsilon|h|^2e^{af},$$ which is equivalent to \eqref{holomorphic} with $b=0$. Hence, by Lemma \ref{lemmaholo}, $\rmd s^2$ is a generalized metric of type $(a,0,c)$ with $\varepsilon=\sgn(K-c)$, since $\Theta$ does not vanish identically. The zeroes of $\sqrt{|K-c|}$ are the zeroes of $\sqrt\Theta$; hence there are $n$ zeroes and their respective orders are $m_1,\dots,m_n$ because of the local expression of $\rmd\sigma_0^2$. This proves that the metric $\rmd s^2=e^{-2\Psi}\rmd\sigma^2$ has the required properties.
 
 We have $\rmd\sigma_0^2=|K-c|^{2/a}\rmd s^2$. So the metric $\rmd s^2$ determines $\rmd\sigma_0^2$ uniquely. On the other hand, in the above construction one may multiply $\rmd\sigma_0^2$ by any positive real number, sufficiently small if $\varepsilon=-1$. In this way one can produce infinitely many generalized Ricci metrics with the required properties but not isometric one to another. 
  \end{proof}

\bibliographystyle{plain}
\bibliography{daniel-zang}

\end{document}